\documentclass{amsart}
\usepackage{MetResMod,amssymb}
\usepackage{verbatim,tikz, gmp}
\usepackage{mathrsfs, amsthm}
\DeclareGraphicsExtensions{.pdf,.png,.mps,.mp}

\datver{0.1B; Revised: 19-1-2014}
\begin{document}
\title[MetResMod]
{Boundary behaviour of Weil--Petersson and fiber metrics for Riemann moduli spaces}


\author{Richard Melrose}
\address{Department of Mathematics, Massachusetts Institute of Technology}
\email{rbm@math.mit.edu}
\author{Xuwen Zhu}
\address{Department of Mathematics, Stanford University}
\email{xuwenzhu@stanford.edu}

\begin{abstract} The Weil--Petersson and Takhtajan--Zograf metrics on the
  Riemann moduli spaces of complex structures for an $n$-fold punctured
  oriented surface of genus $g,$ in the stable range $g+2n>2,$ are shown
  here to have complete asymptotic expansions in terms of Fenchel--Nielsen
  coordinates at the exceptional divisors of the Knudsen--Deligne--Mumford
  compactification. This is accomplished by finding a full expansion for
  the hyperbolic metrics on the fibers of the universal curve as they
  approach the complete metrics on the nodal curves above the exceptional
  divisors and then using a push-forward theorem for conormal
  densities. This refines a two-term expansion due to Obitsu--Wolpert for
  the conformal factor relative to the model plumbing metric which in turn
  refined the bound obtained by Masur. A similar expansion for the Ricci
  metric is also obtained.
\end{abstract}

\maketitle
\tableofcontents

\section*{Introduction}\label{Intro}

The universal curve, $\Cs{g}=\Ms{g,1},$ of Riemann surfaces of genus $g\geq
2$ may be identified with the moduli space of pointed curves as a stack or
as a smooth orbifold fibration $\psi:\Cs{g}\longrightarrow \Ms g$ over the
moduli space (we distinguish notationally between these spaces since later
they have different real resolutions). Deligne and Mumford~\cite{MR0262240} gave
compactifications $\opsi:\oCs{g}=\oMs{g,1}\longrightarrow \oMs{g}$ in which
nodal curves are added covering exceptional divisors corresponding to the
pinching of geodesics on the Riemann surfaces to pairs of nodal points
resulting in a surface, or surfaces, of lower genus but with the same
arithmetic genus. The holomorphic map $\opsi$ then has Lefschetz
singularities and is universal in this sense. Each fiber of $\oMs{g,1}$
carries a unique metric of finite area and curvature $-1,$
complete outside the nodal points.

A resolution of the complex compactification is given below
\begin{equation}
\xymatrix{
\hCs{g}\ar[r]^{\beta}\ar[d]_{\hpsi}&\oMs{g,1}\ar[d]^{\opsi}\\
\hMs g \ar[r]_{\beta}&\oMs{g}
}
\label{MetResMod.2}\end{equation}
in the category of real manifolds with corners (or more correctly tied
orbifolds), which resolves this fiber metric. In particular, the real
fibration in \eqref{MetResMod.2} is a b-fibration in terms of which the
fiber metric is conformal, with a log-smooth conformal factor, to a smooth
metric on a rescaling of the fiber tangent bundle. The resolution
\eqref{MetResMod.2} (in both domain and range) involves a transcendental
step, introducing variables comparable to the length of the shrinking
cycles, i.e.\ Fenchel-Nielsen coordinates.

The regularity properties of the fiber hyperbolic metrics have been widely
studied, see \cite{wolpert1985, MR1037410, MR2039996, MR3077878,
  wolpert2015, MR2067477}, and the results effectively applied, see
\cite{Gell-Redman-Swoboda, Ji-Mazzeo-Muller-Vasy, MR2993753}. The log-smoothness of
the hyperbolic fiber metric up to the boundaries, produced by the real
blow-up, corresponds to a refinement, to infinite order, of the 2-term
expansion obtained by Obitsu and Wolpert \cite{MR2399166} for the fiber
metric relative to the `plumbing metric' on the local model for nodal
curves. The case of a single shrinking geodesic was considered in
\cite{MetLef} and the log-smoothness of the constant curvature fiber
metrics here is proved by an extension of the method used there (which in
structure goes back to Obitsu and Wolpert loc.\ cit.)  We further extend
these results to the case of the universal curve over the moduli space,
$\Ms{g,n},$ of marked Riemann surfaces in the stable case that $2g+n>2.$
The universal curve over $\Ms{g,n}$ may be identified as
$\cC_{g,n}=\Ms{g,n+1}$ but in which one, here by convention the last,
variable is distinguished as the fiber variable in the holomorphic
fibration $\cC_{g,n}\longrightarrow \Ms{g,n}.$

In the unpointed case, let $\cL$ be the fiber tangent bundle of $\psi,$
then the cotangent bundle of $\Ms{g}$ is naturally identified with the
bundle of holomorphic quadratic differentials, i.e.\ holomorphic sections
of $\cL^{-2},$ on the fibers of $\cC_{g}=\Ms{g,1}$
\begin{equation}
q:\Lambda^{1,0}\Ms g\simeq Q\Ms g.
\label{MetResMod.300}\end{equation}
Using this identification, the Weil--Petersson (co-)metric is defined by
\begin{equation}
G_{\WP}(\zeta_1,\zeta_2)=\int_{\fib}\frac{\zeta_1\overline{\zeta_2}}{\mu_H},\
\zeta_1,\ \zeta_2\in Q_m,\ m\in \Ms g
\label{MetResMod.299}\end{equation}
where $\mu_H$ is the area form of the fiber hyperbolic metric and the
integrand itself may be identified as a fiber 2-form.

More generally in the pointed case, the Knudsen-Deligne-Mumford
compactification $\oMs{g,n}$ of the $n$-pointed moduli space may again be
considered as a smooth complex orbifold. The `boundary'
$\oMs{g,n}\setminus\Ms{g,n}$ is a union of normally intersecting, and
self-intersecting, divisors. Expanding on an idea of Robbin and Salamon
\cite{Robbin-Salamon} we extend \eqref{MetResMod.300} to the
compactification, showing that the logarithmic cotangent bundle
$\DL^{1,0}\oMs{g,n},$ with local sections the sheaf of differentials which are
logarithmic across the exceptional divisors, is naturally isomorphic to a
corresponding holomorphic extension of the bundle of holomorphic quadratic
differentials on the fibers of $\oCs{g,n}.$ More precisely,
the projection
\begin{equation}
\opsi:\oCs{g,n}=\oMs{g,n+1}\longrightarrow \oMs{g,n}
\label{MetResMod.306}\end{equation}
is a Lefschetz map (as defined explicitly below) and hence has a
well-defined, and surjective, differential between the logarithmic tangent
bundles. The null bundle, $\oL,$ is thus a holomorphic line bundle over
$\oCs{g,n}$ (with sections being the holomorphic vector fields on the fibers
which vanish at marked points and nodes -- in the marked case nodes also
arise from the collision between, or more precisely the separation of, the
fixed divisors corresponding to the marked points). Then $Q\Ms {g,n}$ extends
as a holomorphic vector bundle $\oQMgn$ where the fiber of $\oQ$ consists
of the holomorphic sections of $\oL^{-2}$ which vanish at marked points
and have consistent values at nodes. That is, if the fibers of $\opsi$ are
`resolved' into a disjoint union of marked Riemann surfaces, by separating
the nodes, then elements of the fiber of $\oQ$ may be interpreted as
meromorphic quadratic differentials in the ordinary sense, with at most
simple poles at marked points and double poles at nodes but where the
double residues at the two points representing a node are the same. Notice that this is meaningful since the double residue, which is the leading coefficient of $f(z)\frac{dz^{2}}{z^{2}}$, is a well-defined complex number. Then $q$
in \eqref{MetResMod.300} extends to a global holomorphic isomorphism
\begin{equation}
\oq:\DL^{1,0}\oMs{g,n}\simeq \oQMgn.
\label{MetResMod.307}\end{equation}
This allows \eqref{MetResMod.299} to be evaluated asymptotically near the
divisors, allowing the full description of the singularities of the
Weil--Petersson co-log-metric, i.e.\ on $\DL^{1,0}\oMs{g,n}.$

One of the properties of the log-cotangent bundle is that
$\DL^{1,0}\oMs{g,n}$ has the cotangent bundle of the divisor (or the
log-cotangent bundle in the case of intersecting divisors) as a subbundle
over the divisor. These `tangential elements' are identified by $\oq$ with the
quadratic differentials with at most simple poles at the corresponding
(separated) nodal points. Thus the quadratic differentials corresponding to
the log-normal directions are the most singular and these produce the
singularities in $G_{\WP}$ as a co-log-metric. Moreover, the restriction to
the log-cotangent bundle of the divisor gives the Weil--Petersson metric
for the finite covering of the divisor as a product of pointed moduli
spaces; this is already noted by Masur \cite{masur1976extension}.

In the pointed case we again have a `metric resolution' to real manifolds
with corners extending \eqref{MetResMod.2} and again involving the
introduction of logarithmic coordinates
\begin{equation}
\xymatrix{
\hCs{g,n}\ar[d]_{\hpsi}\ar[r]^-{\beta}&\oCs{g,n}=\oMs{g,n+1}\ar[d]^{\opsi}\\
\hMs{g,n}\ar[r]^{\beta}&\oMs{g,n}.
}
\label{MetResMod.339}\end{equation}
As remarked above $\hCs{g,n}= \hMs{g,n+1}$ involves an extra step of
resolution compared to $\hMs{g,n}\longrightarrow \oMs{g,n},$ without which
the map $\hpsi$ is not defined. This extra blow-up is of the
codimension-two variety of double points of $\opsi$ which therefore lifts
to a collection of boundary hypersurfaces. As a result the boundary
hypersurfaces of $\hCs{g,n}$ fall into three distinct classes, the `fixed'
hypersurfaces corresponding to the marked points, the `type I' boundary
hypersurfaces corresponding to the resolved singular fibers and the `type
II' boundary hypersurfaces corresponding to the singular set of $\opsi,$ so
the image in the base of the type II hypersurfaces lies within that of the
type I hypersurfaces.

The line bundle $\oL$ lifts to a complex line bundle, $\hL,$ on $\hCs{g,n}$ but
this is not precisely the fiber b-tangent bundle 
\begin{equation}
{}^{\hpsi}T\hCs{g,n}\subset \bT\hCs{g,n}
\label{MetResMod.340}\end{equation}
of $\hpsi,$ although the latter is a well-defined smooth bundle. Namely
$\oL$ is a rescaling of the fiber tangent bundle of $\hpsi$ at the type II
boundary hypersurfaces. Its smooth sections are precisely those of the form
$\rho_{\II}^{-1}V$ where $V$ is a smooth tangential vector field on
$\hCs{g,n}$ which is tangent to the circle fibration of the type II
hypersurface and such that $\hpsi_*(V)=0$ and $\rho_{\II}$ is a collective
defining function for the hypersurfaces of type $\II.$ We think of this as a `cusp
structure'. Extending the special case in \cite{MetLef}:

\begin{theorem}\label{MetResMod.342} The complete metrics of constant
  curvature $-1$ on the fibers of $\Cs{g,n}$ over $\Ms{g,n}$ extend to be
  conformal to a smooth family of fiber Hermitian metrics on the line bundle
  $\hL$ over the fibers of $\hCs{g,n}\longrightarrow \hMs{g,n}$ with a
  positive definite conformal factor which is log-smooth.
\end{theorem}
\noindent A log-smooth function on a manifold with corners is smooth in the
interior and log-smooth at the boundary with an expansion with integer
powers (as for the Taylor series of a smooth function) with logarithmic
factors with powers growing at most linearly. A more detailed description
of the asymptotic expansion is given in~\S\ref{Linearized} and~\S\ref{Conformal}.

It then follows that the Weil--Petersson metric is also log-smooth on
the resolved manifold $\hMs{g,n},$ which corresponds to the introduction of
logarithmic coordinates around the divisors. The passage to logarithmic
variables means that the circle bundles corresponding to the fibrations of
the boundary hypersurfaces of $\hMs{g,n},$ over the divisors, extend off the
boundary to infinite order. In \S\ref{WP} the push-forward theorem from
\cite{MR93i:58148} is applied to \eqref{MetResMod.299} to yield the second
major result of this paper (see also the results of Mazzeo and Swoboda
\cite{Mazzeo-Swoboda}). A holomorphic local defining function for an
exceptional divisor 
\begin{equation}
z_j=\exp(-s_j^{-1}+i\theta_j)
\label{MetResMod.415}\end{equation}
induces a defining function $s_j$ for the corresponding boundary
hypersurface of $\hMs{g,n}$ and also locally trivializes the normal circle bundle.

\begin{theorem}\label{MetResMod.375} On the resolved space $\hMs{g,n}$ the
  Weil--Petersson metric is $\theta$-invariant to infinite order at each of
  the boundary hypersurfaces and near any point in a corner of codimension
  $k,$ takes the form
\begin{equation}
g_{\WP}=\pi\sum\limits_{j=1}^k
s_j\left(\frac{ds_j^2}{s^2_j}+s_j^2d\theta_j^2\right)+g'_{\WP},\
g'_{\WP}(\pa_{\theta_j},\cdot)=O(s_j^4),
\label{MetResMod.301}\end{equation}
in terms of \eqref{MetResMod.415} for local holomorphic defining functions
for the divisors. Here $g'_{\WP}$ is a log-smooth hermitian tensor which
restricts to the corner to be the lift of the Weil--Petersson metric on the
$k$-fold intersection of divisors.
\end{theorem}

The leading order part of the metric in~\eqref{MetResMod.301} is the same
as the model metric given by Yamada~\cite{MR2067477}. Note that the local
fiber differentials $d\theta_j$ may be replaced by connections forms, in
fact it is natural to take (extensions of) the connections forms, $\alpha _j,$
fixed by holomorphy and the hermitian structures on the normal bundles. Then
\begin{equation}
g_{\WP}=\pi\sum\limits_{j=1}^k
s_j\left(\frac{ds_j^2}{s^2_j}+s_j^2\alpha _j^2\right)+g''_{\WP}
\label{MetResMod.418}\end{equation}
where $g''_{\WP}$ has the same restriction properties as $g'_{\WP}.$ It is
also permissible to replace the $s_j$ by $h_j=s_j+O(s_j^2)$ where the $h_j$
are the length functions arising from the hermitian structures on the
normal bundles, without changing the conclusion regarding the remainder
term (in fact the proof of \eqref{MetResMod.301} passes through this change
of variable).

Given the regularity results below for the lengths of the short closed
geodesics, the same form occurs in Fenchel--Nielsen coordinates, i.e.\ with
the $s_j$ interpreted as the normalized lengths of the nearby shrinking geodesics for
the non-fixed divisors. The tangential metrics come from the covering of
the components of the $k$-fold intersection of divisors by a product of
pointed moduli spaces. For $\Ms g$ such an asymptotic expansion has been
deduced, from~\cite{MetLef}, using related methods by Mazzeo and Swoboda
\cite{Mazzeo-Swoboda}.

As a corollary, in~\S\ref{SG} we derive the formula for the length of the shrinking closed geodesic 
near the $j$-th non-fixed divisor in terms of the logarithmic coordinates
$s_j=1/\log(|z_j|^{-1})$ for which we use the abbreviation $\ilog|z_j|.$ 
\begin{corollary}
The length of the shrinking geodesic is a log-smooth function of $s_{j}$, and has the form
\begin{equation}
L_{j}(s_{j})=2\pi^{2}s_{j}(1+s_{j}e(s_{j}))
\end{equation}
where $e$ is a log-smooth function up to the boundaries.
\end{corollary}
In the paper \cite{Wolf-Wolpert} of Wolf and Wolpert, it is claimed that
$L_j(s_{j})$ is real-analytic in $s_j,$ which would preclude the appearance
of logarithmic terms in $e(s_j)$ (as well as being a stronger analytic
statement). However there is an error in the bound deduced from equation
(2.2) in \cite{Wolf-Wolpert} which appears to invalidate the argument.

From the asymptotic expansion of the Weil--Petersson metric it follows that
the Ricci curvature has a similar expansion, see
\eqref{MetResMod.390}. Trapani in \cite{MR1176026} showed that the
corresponding Ricci metric is complete. In~\S\ref{Ricci} we show that this metric is locally
equal, to leading order, to a product of cusp metrics near the intersection
of divisors; this refines a result of Liu, Sun and Yau in \cite{MR2169586}.

The curvature tensor of the Weil--Petersson metric is also computed and
specifically the decay rates of the sectional curvature along the normal
and tangential directions near the divisors are given,
see~\eqref{MetResMod.407}.

The curvature form of the fiber hyperbolic metric on the vertical tangent
bundle of $\psi$ over $\cM_g$ was computed by Wolpert \cite{MR842050} who
showed that it pushes forward to a multiple of the K\"ahler form of the
Weil--Petersson metric. Reinterpreting this as a local index theorem,
Takhtajan and Zograf in \cite{MR1101693} extended the result to the pointed moduli
space $\Ms{g,n},$ finding extra `boundary terms' in the push-forward
as an additional K\"ahler form. This metric is given, as a cometric lifted using
the Weil--Petersson metric, by a sum over the fixed divisors in $\Ms{g,n+1}:$
\begin{equation}
G_{\TZ}(\zeta_1,\zeta_2)=\sum\limits_{j}\int_{\fib}E_j\frac{\zeta_1\overline{\zeta_2}}{\mu_H},\
\zeta_1,\ \zeta_2\in Q_p,\ p\in \Ms{g,n}.
\label{MetResMod.303}\end{equation}
Here $E_j$ is a boundary forcing term, the solution of $(\Lap+2)E_j=0$
which is in $L^2$ on the fibers including up to the marked points and nodes, except for a prescribed singularity
(corresponding to the non-$L^2$ formal solution) at the point corresponding to the $j$th
fixed divisor denoted as $F_{j}$; 
in \cite{MR1101693} it is obtained as an Eisenstein series. Such a function
is well-defined on any stable Riemann surface with cusps; it is strictly
positive away from the cusps and, with these resolved to boundaries, $E_j$
is smooth up to, and vanishes simply at each cusp boundary except the
`forcing boundary' where it has a singularity $s^{-2}$ where the metric is
locally $\frac{ds^2}{s^2}+s^2d\theta^2.$ The asymptotic behaviour of
$G_{\TZ}$ is determined by the structure of the $E_j.$ In~\cite{Park-Takhtajan-Teo} the K\"ahler potential and Chern forms of this metric were calculated.

Each boundary hypersurface of $\hMs{g,n}$ corresponds to either two or three boundary
hypersurfaces in the resolved universal curve $\hCs{g,n},$ depending on
whether the node to which this gives rise disconnects the Riemann surface
or not. These have interior fibers which are one or two connected Riemann
surfaces with two nodes and a cylindrical `neck' joining the nodes. The
behavior of $E_j$ is slightly different in the two cases. If the Riemann
surface remains connected without the neck, then $E_j$ approaches the
corresponding boundary forcing term for this Riemann surface and vanishes
simply at the neck with coefficient being a bridging function discussed in
the body of the paper. In case of separation into two Riemann surfaces
again $E_j$ approaches the `local' $E_j$ for the component that meets
$F_j,$ it vanishes simply on the neck and vanishes to fourth order at the
second component with leading coefficient which is the product of a scattering (or
$L$-function) constant and the singular boundary term $E_*$ on this Riemann
surface with pole at the node where it meets the neck. In all cases $E_j$
is globally log-smooth once multiplied by the square of a defining function
for $F_j.$

This description can be iterated to determine the precise leading term of
$E_j$ at a boundary surface of codimension $k$ in $\hMs{g,n}.$ This
corresponds to the intersection of $k$ divisors in $\oMs{g,n}$ and the fiber
is the initial Riemann surface subject to $k$ degenerations, each either
the shrinking of a geodesic or the `bubbling off' of a sphere due to the
collision of marked points (or marked points with nodes). In all cases the
fixed divisor $F_j$ meets one of the component surfaces and $E_j$
approaches the corresponding boundary forcing term there. Any other 
component is connected to $F_j$ through one or more paths, passing through
a sequence of nodes and necks. Consider those paths which pass through the
minimum number, $\sigma,$ of necks. Each of these gives rise to part of the
leading term of $E_j$ at the Riemann surface in question; it vanishes to
order $4\sigma$ there with coefficient the $E_*$ for that Riemann surface with
pole at the node through which the path entered, and another constant
coefficient formed by a product of scattering ($L$-function) factors
corresponding to the sequence of nodes through which the path passes; thus
the leading term of $E_j$ is in general a sum of such terms but all are
positive. At the necks essentially the same conclusion holds except that
the order of vanishing is $4\sigma+1$ with the coefficient a bridging
function.

The asymptotic behavior of $E_j$ at the boundary of $\hMs{g,n}$ leads to a
corresponding asymptotic expansion for the Takhtajan--Zograf metric,
analogous to \eqref{MetResMod.301}, but with combinatorial complications;
this refines results of Obitsu, To and Weng \cite{MR2443304}. At the
interior of a boundary hypersurface of $\hMs{g,n}$ the behavior is
relatively simple if the corresponding divisor lifts from $\oMs{g,n}.$
Namely the divisor itself has a local covering by either a moduli space
$\oMs{g-1,n+2}$ if the node does not separate and otherwise by some product
$\oMs{g_1,n_1+1}\times \oMs{g_2,n_2+1},$ $n_1+n_2=n,$ $g_1+g_2=g-1.$ Then
\begin{equation}
\begin{gathered}
g_{\TZ}=h(ds^2+s^{4}d\theta^2)+sh_1
g_{\TZ,g_1,n_1}+sh_2g_{\TZ,g_2,n_2}+O(s^2),\ n_1,n_2>0\\
g_{\TZ}=h(ds^2+s^{4}d\theta^2)+sh_1
g_{\TZ,g_1,n}+s^4h_2g_{\TZ, g_2,1}+O(s^2),\ n_1=n
\end{gathered}
\label{MetResMod.308}\end{equation}
has log-smooth coefficients. Here, the coefficients $h,$ $h_1$ and $h_2$
are $\theta-$invariant to all orders in $s$ and positive but not constant. We do
not completely explore the asymptotics of $g_{\TZ},$ but it is bounded
above by a multiple of the Weil--Petersson metric and always vanishes
relative to it in normal directions to the boundary faces; the same is
therefore true of the corresponding K\"ahler form.

The authors would like to acknowledge helpful conversations with Mike
Artin, Rafe Mazzeo, David Mumford, Jan Swoboda, Scott Wolpert and Mike Wolf
and also Semyon Dyatlov for comments on the manuscript and assistance with
the figures. We would also like to thank the referee for a careful reading
and valuable comments.

\paperbody

\section{Lefschetz maps}\label{Lefschetz}

We consider a complex manifold with normally intersecting and
self-intersecting divisors, $(C,G_*).$ Thus the $\{G_i\}$ are a finite
collection of closed immersed connected complex hypersurfaces and near each
point of $C$ there are \emph{admissible coordinates} $(z_*,\tau_*)$ where
the $z_l$ define the \emph{local} divisors passing through the point as
$\{z_l=0\}.$ On such a manifold there is a well-defined `logarithmic'
complex tangent bundle $\DT^{1,0}C$ and corresponding cotangent bundle
$\DL^{1,0}C$ determined by the $G_*.$ Namely the spaces of locally
holomorphic sections of $\DT^{1,0}C$ are the holomorphic vector fields
which are tangent to all the local divisors. In admissible coordinates
$\DT^{1,0}C$ is spanned by the holomorphic vector fields $z_l\pa_{z_l}$
and $\pa_{\tau_k}.$ The complex dual of this bundle, $\DL^{1,0}C,$ is
locally spanned in these coordinates by the $dz_l/z_l$ and $d\tau_k.$ The universal curve $\oCs{g,n}$ has such a structure (except for the orbifold points). And there is a natural map from  $\oCs{g,n}$ to the compactified moduli space $\oMs{g,n}$ which will be discussed below.

We consider Lefschetz maps, which have singularities modelled on the
`plumbing variety'
\begin{equation}
\phi:\bbC^2\ni(z,w)\longmapsto t=zw\in\bbC.
\label{MetResMod.282}\end{equation}

\begin{definition}\label{MetResMod.266} A \emph{Lefschetz map}
  $\phi:C\longrightarrow M$ is a holomorphic map between complex manifolds
  (subsequently orbifolds) with the following properties
\begin{enumerate} 
\item The fiber dimension is one: $\dim_{\bbC}C=\dim_{\bbC}M+1.$
\item The fibers of $\phi$ are compact.
\item  Both domain and range carry normally intersecting (and self-intersecting) divisors
  which will be denoted $(C,G_*,F_*)$ and $(M,G_*').$
\item The $F_l,$ $l=1,\dots,n$ are `fixed divisors', without self-intersections or
  intersections with the other $F_*$ and such that, for each $i,$
\begin{equation}
\phi:F_i\longrightarrow M\text{ is a biholomorphism.}
\label{MetResMod.317}\end{equation}
\item The other divisors, $G_i,$ in $C$ map onto the divisors in $M:$ 
\begin{equation}
\phi(G_i)=G_{i}'.
\label{MetResMod.318}\end{equation}
\item The $G_i$ have self-intersections, $S_i,$ of codimension two (if they exist) and for different $i$, these are also
  disjoint in $C;$ outside the $S_i,$ $\phi$ has surjective differential. 
\item Near each double point $p\in S_{i}\subset G_i$ there are admissible coordinates
  $z,w,z_{r},\tau_j$ in $C$ and coordinates $t,z_{r}',\tau_j'$ near the image of $p$ in
  $M$ such that 
\begin{equation}
\phi^*t=zw,\ \phi^*\tau_j'=\tau_j,\ \phi^*z_r'=z_r,
\label{MetResMod.319}\end{equation}
where $G_i$ is locally defined by $\{z=0\}\cup\{w=0\}$ and the $z_r'$ define the local
divisors through $p$ other than $G'_{i}$ which is defined by $t=0.$ 
\end{enumerate}
It follows that each of the fibers is a nodal Riemann surface, with marked
points from the intersections with the $F_l.$ The Lefschetz map is said to
be \emph{stable} if each of the component Riemann surfaces (when the nodal
points are separated) in the fibers are stable; i.e. for each component the genus $g$ and the total number of nodes and punctures $n$ satisfy $2g-2+n>0.$
\end{definition}

Consider the invariance properties of the local form of the `plumbing'
model for a Lefschetz map as in \eqref{MetResMod.282} near the singular
surfaces $S_i:$ 
\begin{equation}
\{(z,w)\in\bbC;|z|,|w|<\delta \}\ni (z,w)\longmapsto zw\in\bbC.
\label{9.3.2017.1}\end{equation}
This normal form can be regained after separate holomorphic changes of
coordinates in the disk $z=0$ and $w=0.$

\begin{lemma}\label{9.3.2017.7} Suppose that
\begin{equation}
z\longmapsto z(1+zf(z)),\ w \longmapsto w(1+wg(w))
\label{9.3.2017.2}\end{equation}
are separate holomorphic coordinate changes fixing the origins, then there
are holomorphic functions (germs near the origin) $a(w),$ $b(z)$ and
$e(z,w)$ such that
\begin{multline}
Z=z(1+zf(z)+wa(w)+zwe(z,w)),\\ W=w(1+wg(w)+zb(z)+zwe(z,w))
\text{ satisfy } ZW=t.
\label{9.3.2017.3}\end{multline}
\end{lemma}

\begin{proof} 
The identity we aim for is
\begin{equation}
zw(1+zf(z)+wa(w)+wze(z,w))(1+wg(w)+zb(z)+wze(z,w))=zw.
\label{9.3.2017.4}\end{equation}
Cancelling factors and expanding out this becomes 
\begin{multline}
wg(w)+zb(z)+zwe(z,w)
+zf(z)+wa(w)+zwe(z,w)+zwf(z)g(w)\\
+z^2b(z)f(z)+z^2wf(z)e(z,w)
+w^2g(w)a(w)+zwa(w)b(z)+zw^2a(w)e(z,w)\\
+w^2zg(w)e(z,w)+z^2wb(z)e(z,w)+z^2w^2e^2(z,w)
=zwh(zw)
\label{9.3.2017.5}\end{multline}
Separating out the `pure terms' it follows that \eqref{9.3.2017.5} is a
consequence of demanding
\begin{equation}
\begin{gathered}
w\left(g(w)+a(w)+wg(w)a(w)\right)=0\Longleftrightarrow a(w)=-(1+wg(w))^{-1}g(w)
\\
z\left(b(z)+f(z)+zb(z)f(z)\right)=0\Longleftrightarrow b(z)=-(1+zf(z))^{-1}f(z)
\\
\begin{aligned}
zw\big(2e(z,w)+f(z)g(w)+&zf(z)e(z,w)+a(w)b(z)+wa(w)e(z,w)\\
+&wg(w)e(z,w)+zb(z)e(z,w)+zwe^2(z,w))\big)=0.
\end{aligned}
\end{gathered}
\label{9.3.2017.6}\end{equation}
By the implicit function theorem the last equation has a unique, and holomorphic, solution
with $e(z,w)$ close to $-\ha f(z)g(w)$ for $z,$ $w$ small.
\end{proof}
\noindent Holomorphic
dependence on parameters follows from the same argument, so the result carries over to the
case of a Lefschetz map with base of dimension greater than one, in the
sense above, locally near each singular surface.

\begin{corollary}\label{MetResMod.410} For a stable Lefschetz fibration
  there are holomorphic coordinates near each singular point in terms of
  which both the Lefschetz map and the family of hyperbolic metrics on the
  singular fibre are in normal form.
\end{corollary}

Conversely:
\begin{lemma}\label{9.3.2017.8} Any local holomorphic coordinate transformation
  in $z,$ $w$ near a singular point which leaves the form of the family of
  hyperbolic metrics and the coordinate form of the Lefschetz
  map unchanged can only change the defining function for the divisor in
  the base from $t$ to $ct(1+th(t))$ where $|c|=1.$
\end{lemma}

\begin{proof} If a coordinate transformation preserves the normal form for
  the  Lefschetz map it must fix the singular point and hence the preimage
  of its image in the base. Thus it must map the surfaces $z=0$ and $w=0$
  into themselves, or each other. The latter possibility  can be ignored,
  since we may simply exchange $z$ and $w$ and preserve the form. Thus the
  normal form of the hyperbolic metric near the ends fixes the conformal
  structure near the end and hence the coordinates up to a constant factor
  of norm $1.$ The coordinate transformation on the total space must therefore be
  of the form $z\longmapsto e^{i\theta}z(1+zf(z)+wF(z,w)),$ $w\longmapsto
  e^{i\theta'}w(1+wg(w)+zG(z,w))$ with $\theta$ and $\theta'$ real. Thus
\begin{equation*}
t\longmapsto e^{i(\theta+\theta')}t(1+O(z,w))=ct(1+th(t))\Longrightarrow |c|=1. 
\label{9.3.2017.9}\end{equation*}
\end{proof}

\begin{proposition}\label{MetResMod.411} The divisors $G'_i$ in the base of a
  stable Lefschetz map have natural hermitian structures on their normal bundles.
\end{proposition}

\begin{proof} The possible changes of local holomorphic defining function
  for a divisor $G^*_i$ in the base, in terms of which the Lefschetz map
  takes normal form near the corresponding $S_i,$ are limited to have
  differential of norm one on the divisor and hence induce a hermitian
  structure.
\end{proof}

As the name indicates, if one fiber of a Lefschetz map is stable then the
map is stable in a neighborhood of that fiber. In the case of a singular
fiber the arithmetic genus $g_a$ is the sum of the number of pairs of nodal
points and the genus from each of the components of the nodal surface with
nodes separated and it follows that $g_a+2n>2$ is then constant.

Away from the self-intersections, $S_i,$ of the $G_i,$ $\phi$ is a
holomorphic submersion mapping the divisors $G_i$ onto the $G'_{i}$ so the
logarithmic differentials $dz'_{i}/z'_{i}$ pull back to be $dz_i/z_i.$ Near
$p\in S_i$ this remains true for the divisors other than $G_i',$ locally
defined by $t$ which, by \eqref{MetResMod.319} satisfies 
\begin{equation}
\phi^*(dt/t)=dz/z+dw/w.
\label{MetResMod.320}\end{equation}
So at these points, $\phi^*$ maps $\DL^{1,0}_{\phi(p)}M$ injectively to $\DL^{1,0}_{p}C$ 
and hence there is a well-defined dual log-differential which is still surjective
\begin{equation}
\phi_*:\DT^{1,0}_pC\longrightarrow \DT^{1,0}_{\phi(p)}M,\ \forall\ p\in C.
\label{MetResMod.268}\end{equation}
In this sense a Lefschetz map is a `log fibration' (the complex analog of
the b-fibrations considered in the real case below).

\begin{lemma}\label{MetResMod.321} The null bundle of the logarithmic differential
  $L=L_{\phi}\subset\DT^{1,0}C$ is a holomorphic subbundle which
  reduces to the fiber tangent bundle at regular points.
\end{lemma}

\begin{proof} This is immediate from the local form of $\phi$ required in
  the definition. Near regular points of the map. the fiber is smooth
  and $L$ is spanned by a non-vanishing holomorphic vector field tangent to
  the fibers. Near singular points there are coordinates as in
  \eqref{MetResMod.319} and the log-differential has null space spanned by
  $z\pa_z-w\pa_w.$ 
\end{proof}

The fiber $\dbar$-operator on regular fibers is naturally a differential
operator 
\begin{equation*}
\dbar:\CI(C_{\reg};L)\longrightarrow \CI(C_{\reg};L\otimes \overline{L}^{-1}).
\label{MetResMod.322}\end{equation*}

\begin{lemma}\label{MetResMod.323} The fiber $\dbar$-operator extends
  smoothly to a `log-differential' operator 
\begin{equation}
\Lbar:\CI(C;L)\longrightarrow \{u\in\CI(C;L\otimes
\overline{L}^{-1});u=0\Mat \bigcup_iS_i\cup\bigcup_l F_l\}.
\label{MetResMod.324}\end{equation}
\end{lemma}

\begin{proof} This is clear away from the singular points $S_i\subset C.$
  Near each such point in local coordinates \eqref{MetResMod.319}, 
\begin{equation}
\Lbar(a(z\pa_z-w\pa_w))=\ha
\left({\bar z\pa_{\bar z}-\bar w\pa_{\bar w}} a\right)(z\pa_z-w\pa_w)\cdot
(\frac{d\bar z}{\bar z}+\frac{d\bar w}{\bar w})
\label{MetResMod.325}\end{equation}
where the two local components of the singular fiber are $z=0$ and $w=0$
and the coefficient vanishes at the $S_i.$ 
\end{proof}

The stability assumption on the fibers and an application of Riemann--Roch
show that $\Lbar$ is injective on each fiber since the null space consists
of holomorphic vector fields vanishing at the fixed divisors and, on the
$G_i$ at the $S_i.$

As already noted the essential property of the Knudsen-Deligne-Mumford Lefschetz map
$\oCs{g,n}\longrightarrow \oMs{g,n}$ is that it is universal for stable
Lefschetz maps. Namely

\begin{theorem}[Knudsen--Deligne--Mumford\cite{MR0262240,
      MR702953}]\label{MetResMod.327} For any stable Lefschetz map, in the
  sense of Definition~\ref{MetResMod.266}, there is a unique commuting
  square of holomorphic maps
\begin{equation}
\xymatrix{
C\ar[r]^{\chi^{\#}}\ar[d]_{\phi}&\oCs{g,n}\ar[d]^{\opsi}\\
M\ar[r]^{\chi}&\oMs{g,n}
}
\label{MetResMod.326}\end{equation}
where $\chi^{\#}$ is a fiber biholomorphism.
\end{theorem}

In \cite{Robbin-Salamon}, Robbin and Salamon give an infinitesimal criterion
for such universality at the germ level. We proceed to review this result.

The operator $\Lbar$ in \eqref{MetResMod.324} is not surjective, in fact
again by Riemann-Roch its image has complement of dimension $3g_a-3+n.$ As
noted above, a Lefschetz map has surjective log-differential. It follows
that any smooth log vector field (i.e.\ tangent to the divisors) on the
base is $\phi$-related to such a smooth vector field on $C.$ Consider the
sheaf over $M$ with sections over an open set $O\subset M$
\begin{multline}
\cE(O)\\
=\{V\in\CI(\phi^{-1}(O);T^{1,0}C):\exists\ v\in\CI(O;T^{1,0}M);\dbar
v=0,\ \phi_{p*}V(p)=v(\phi(p))\}
\label{MetResMod.328}\end{multline}
consisting of the vector fields which are $\phi$-related to a holomorphic
vector field on $O.$ As the null space of $\phi_*$
\begin{equation}
\CI(\phi^{-1}(O);L)\subset\cE(O)
\label{MetResMod.329}\end{equation}
with the quotient being the holomorphic vector fields on $O.$

\begin{lemma}\label{MetResMod.331} For a Lefschetz map the operator $\Lbar$
  extends to $\cE:$
\begin{equation}
\Lbar:\cE(O)\longrightarrow \{u\in\CI(\phi^{-1}(O);L\otimes \overline{L}^{-1});
u=0\Mat \bigcup_iS_i\cup\bigcup_l F_l\}.
\label{MetResMod.330}\end{equation}
\end{lemma}

\begin{proof} On a coordinate patch $U\subset\phi^{-1}(O)\subset C,$ as in
  the definition of a Lefschetz map, an element $u\in\cE(O)$ restricts to
  be of the form
\begin{equation}
u=v+w,\ w\in\CI(U;L)
\label{MetResMod.332}\end{equation}
and defining $\Lbar u=\Lbar w$ is independent of choices.
\end{proof}

Now, the result alluded to above is

\begin{proposition}[Robbin--Salamon
    \cite{Robbin-Salamon}]\label{MetResMod.333} A Lefschetz map is
  universal (at the germ level) at a given fiber if and only if $\Lbar$ in
  \eqref{MetResMod.330} is an isomorphism.
\end{proposition}

\begin{proof} See \cite{Robbin-Salamon}. A proof using techniques much
  closer to those used here can also be constructed.
\end{proof}

Note that Robbin and Salamon in \cite{Robbin-Salamon} proceed to construct
such germs of universal Lefschetz fibrations (as `unfoldings' of the
central fiber) and use these to (re-)construct the Knudsen-Deligne-Mumford
compactification.

\begin{proposition}\label{MetResMod.349} For each fiber of $\phi$, sections of the bundle
  $L\otimes\overline{L}^{-1}$ which vanish at the $F_l$ and $S_i$ may be
  paired with `quadratic differentials' to give a natural and
  non-degenerate complex pairing
\begin{multline}
\{u\in\CI(\phi^{-1}(O);L\otimes\overline{L}^{-1})\big|_{\phi^{-1}(m)};
u=0\Mat \bigcup_iS_i\Mand\bigcup_l F_l\}\\
\times\{q\in\CI(\phi^{-1}(O);L^{-2})\big|_{\phi^{-1}(m)};q=0\Mat\bigcup_lF_l\}\\
\ni (u,q)\longmapsto\int_{\fib}(u\cdot q)\in\bbC.
\label{MetResMod.334}\end{multline}
\end{proposition}

\begin{proof} At regular fibers, away from the $F_l,$ the sections are of
  the form $u=u' \pa_z\cdot\overline{dz}$ and $q=q'dz^2$ with $u'$ and $q'$
  smooth. The complex pairing of $\pa_z$ and $dz$ allows this to be
  interpreted as a local area form 
\begin{equation}
u'q'\overline{dz}\cdot dz.
\label{MetResMod.350}\end{equation}

Near a fixed divisor with $z$ and admissible fiber coordinate $u=u'
z\pa_z\cdot{d\bar z/\bar z}$ and $q=q'(dz/z)^2$ where by hypothesis $u'$
vanishes at $z=0$ as does $q'.$ Working in polar coordinates
$z=re^{i\theta}$ the area form in \eqref{MetResMod.350} becomes
\begin{equation}
u'q'{\frac{d\bar z}{\bar z}}\cdot \frac{dz}z=u'q'\frac{dr}rd\theta.
\label{MetResMod.351}\end{equation}
Since both $u'$ and $q'$ vanish at $r=0$ this is a multiple of the standard
area form in polar coordinates $rdrd\theta$ and so the integral pairing
extends across these divisors.

On the singular fibers the same discussion applies near the fixed
divisors. Near a nodal point, in some $S_i,$ there are Lefschetz coordinates
$z$ and $w.$ By assumption
$u=u'(z\pa_z-w\pa_w)\cdot(\frac{d\bar z}{\bar z}+\frac{d\bar w}{\bar w})$ with $u'$
  vanishing at $z=0$ or $w=0$ on the two intersecting parts of the
  fiber. Similarly  
\begin{equation}
q=q'(\frac{dz}z+\frac{dw}w)^2
\label{MetResMod.352}\end{equation}
where the compatibility condition on $q'$ is that it have a well-defined
value at $z=w=0,$ approached continuously from both local parts of the
fiber. Thus the local area form given by the pairing becomes
\begin{equation}
u'q'\frac{dr}rd\theta
\label{MetResMod.353}\end{equation}
which is bounded by $Cdrd\theta$ locally. The resulting pairing integral is
finite, but more significantly the leading terms in the integral at $w=0$
and $z=0$ cancel.

That the resulting pairing is non-degenerate on the smooth sections is then
clear.
\end{proof}

Although somewhat ad hoc in appearance this pairing actually corresponds to
a natural distributional pairing on real resolved spaces.

\begin{proposition}\label{MetResMod.335} In terms of the pairing
  \eqref{MetResMod.334}, the range of $\Lbar$ in \eqref{MetResMod.324} is naturally
  identified, over each fiber with the annihilator of 
\begin{equation}
\oQ(m)=\{\zeta\in\CI(\pi^{-1}(m);L^{-2});\zeta=0\Mat\bigcup_lF_l\Mand\dbar\zeta=0\}.
\label{MetResMod.336}\end{equation}
\end{proposition}

\begin{proof} The range of $\Lbar$ in \eqref{MetResMod.325} is closed with
  finite-dimensional complement for each fiber since the domain differs
  from the standard domain on the disjoint union of the Riemann surfaces,
  into which each fiber decomposed, by a finite-dimensional space. The
  restriction of the pairing to a finite dimensional complement and to an
  appropriate finite-dimensional subspace of the space of quadratic
  differentials in \eqref{MetResMod.334} is therefore non-degenerate. On
  the other hand the space of holomorphic quadratic differentials, in this
  sense, pairs to zero with the range of $\Lbar$ and has the same
  dimension as the complementary space so is indeed the annihilator of the
  range.
\end{proof}

Note there is a certain inconsistency in notation regarding the extension of $Q$ to
the compactification by $\oQ$, since following the convention of denoting
extensions to the compactification by a `bar' would conflict with the notation for
complex conjugation.

That $\oQ$ is a holomorphic bundle over $M$ is again a consequence
of stability, that its rank is constant. This follows from
algebraic-geometric arguments or from the known holomorphy of the log
cotangent bundle to $\oMs{g,n}$ and the pointwise identification of it with
$\oQ.$

Now the extended operator $\Lbar$ in \eqref{MetResMod.330} defines a bundle
map 
\begin{equation}
\DT^{1,0}M\longrightarrow (\oQ)',\ \DL^{1,0}M\longrightarrow \oQ
\label{MetResMod.337}\end{equation}
and Proposition~\ref{MetResMod.333} of Robbin-Salamon asserts that
this is an isomorphism, precisely when the Lefschetz map is universal at a
germ level. This is the identification $\oq$ of \eqref{MetResMod.307} in the case of
$\oMs{g,n}.$

\section{Real and metric resolutions}\label{Real-resolution}

Although regular, or at least minimally singular, in themselves (in
particular `flat'), it is necessary for analytic purposes to resolve the
Lefschetz fibrations considered above. The construction here is a direct
generalization of the constructions in \cite{MetLef}. Although it is not
discussed here, there is a simpler (non-logarithmic) resolution which is
appropriate for the analysis of the fiber $\Lbar.$

The first step in the resolution is to blow up, in the real sense, the
divisors, in both domain and range. Since any intersections or
self-intersections are normal, i.e.\ transversal, there is no ambiguity in
terms of the order chosen in doing this. Since we insist, as a matter of
definition or notation, that the boundary hypersurfaces of a manifold with
corners be embedded, and this need not be the case here, the result may
only be a \emph{tied manifold} -- a smooth manifold locally modelled on the
products $[0,\infty)^k\times(-\infty,\infty)^{n-k}$ but with possibly
non-embedded boundary hypersurfaces.

\begin{definition}\label{MetResMod.270} If $(M,G_*)$ is a complex manifold
  with normally intersecting divisors, then the real resolution 
\begin{equation}
\Mt=[M;G_*]_{\bbR}
\label{MetResMod.271}\end{equation}
is a tied manifold (so with corners).
\end{definition}

Locally the complex variables $z_l$ defining the divisors lift to
$r_le^{i\theta_l}$ in the real blow-up so the boundaries carry circle
fibrations. The logarithmic tangent and cotangent bundles lift to be
canonically isomorphic to the corresponding b-tangent and cotangent
bundles, which therefore carry induced complex structures; namely 
\begin{equation}
z\pa_z\text{ lifts to }r\pa_r+i\pa_{\theta}.
\label{MetResMod.295}\end{equation}

This definition can be applied to both the domain and the range for
the `plumbing variety' model for a Lefschetz map
\eqref{MetResMod.282}. Real blow-up of the two divisors $z=0$ and $w=0,$ 
introduces polar coordinates $z=r_1e^{i\theta_1}$ and $w=r_2e^{i\theta_2}$
and similarly blow up of $t=0$ introduces $t=re^{i\theta}$ in the range space.
The local Lefschetz maps then lifts to be a fibration in the angular
variables with a simple b-fibration condition satisfied in the radial variables
\begin{equation}
\begin{gathered}
\tphi(r_1,r_2,\theta_1,\theta_2)=(r,\theta)\\
r=r_1r_2,\ \theta=\theta_1+\theta_2.
\end{gathered}
\label{MetResMod.294}\end{equation}
Thus the lifted map is indeed a b-fibration. Globalizing this statement gives:

\begin{proposition}\label{MetResMod.343} A Lefschetz map between complex
  manifolds with divisors lifts to a b-fibration 
\begin{equation}
\xymatrix{
\tC\ar[d]_{{\beta}_{\bbR}}\ar[r]^{\tphi}&\Mt\ar[d]^{{\beta}_{\bbR}}\\
C\ar[r]^{\phi}&M
}
\label{MetResMod.272}\end{equation}
which is `simple' in the sense that it is the real analog of a Lefschetz
map and each boundary defining function in the base lifts, locally, to
be a product of at most two factors as in \eqref{MetResMod.294}.
\end{proposition}

However, to resolve the metric we need to go further.

\begin{definition}\label{MetResMod.274} The metric resolution of a
  Lefschetz fibration is defined from the real resolution in
  Definition~\ref{MetResMod.270} by 
\begin{enumerate}
\item Logarithmic resolution of the boundary hypersurfaces in both domain
  and range, i.e.\ by introduction of the function $\ilog\rho=1/\log(1/\rho)$
  in place of each (local) boundary defining function $\rho$ . This results
  in tied manifolds in domain and range mapping smoothly, and
  homeomorphically, to the resolutions in Definition~\ref{MetResMod.270};
  the resolved range space is denote $\hM.$
\item Further radial blow-up, in the domain, of the preimages after the
  first step of the boundary faces of codimension-two in
  Definition~\ref{MetResMod.272} resulting in the tied manifold with
  corners manifolds $\hC.$
\end{enumerate}
\end{definition}

\noindent To illustrate the resolution, we give the example when there is one nodal intersection explicitly in terms of local coordinates. Locally the fibration is given by the following plumbing variety
$$
\begin{gathered}
P=\{(z,w)\in\bbC^2;\ \exists\ t\in\bbC, \ zw=t,\ |z|\le\tq,\ |w|\le\tq,\ |t|\le\ha\},\\
P\overset{\phi}\longrightarrow\bbD_{\ha}=\{t\in\bbC;|t|\le\ha\}.
\end{gathered}
$$
The real resolution introduces polar coordinates 
$$
z=r_{z}e^{i\theta_{z}}, \ w=r_{w}e^{i\theta_{w}}, \ t=r_{t}e^{i\theta_{t}}.
$$
Step (1) above, the logarithmic resolution, introduces variables $\ilog\rho=1/\log(1/\rho)$ for the radial variables
$$
s_{z}=\ilog r_{z}, \ s_{w}=\ilog r_{w}, \ s=\ilog r_{t}.
$$ 
Step (2), the radial blow-up, resolves the singularity in
$s=\frac{s_{z}s_{w}}{s_{z}+s_{w}}$ by introducing polar coordinates for
$(s_{z},s_{w}):$
$$
(s_{z},s_{w})=(R R_{z}, R R_{w}), \ R=\sqrt{s_{z}^{2}+s_{w}^{2}}.
$$
Figure~\ref{MetResMod.419} indicate the resolution of the domain space. The first
step illustrates the blow-up of complex divisors, each becoming a boundary
hypersurface (with circle bundle). The second `trivial' step on the left
corresponds to the replacement of the radial defining functions by their
(doubly inverted) logarithms. The radial blow up amounts to the
introduction of polar coordinates around the corners which correspond to
the nodal surfaces.

\begin{figure}[htp]
    \centering
    \includegraphics[width = 0.6\textwidth]{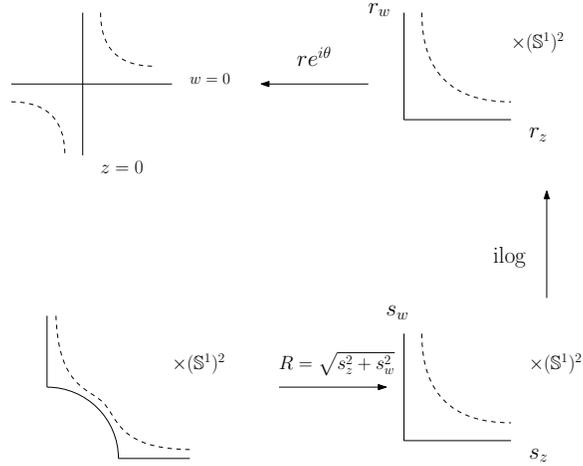}
    \caption{Local form of the metric resolution}
 \label{MetResMod.419} \end{figure}

In the case of a multi-Lefschetz fibration, the local coordinate description
(with more parameters) is the same but in various factors. Performing the
two operations (1) and (2) in the opposite order is by no means
equivalent. Although we use the same notation for the resolutions in domain
and range note that an extra step is involved in the domain, where the
codimension two intersections of the divisors corresponding to the nodal
surfaces are resolved to become boundary hypersurfaces. It is for this
reason that we distinguish notationally between $\Cs{g,n}$ and
$\Ms{g,n+1},$ which are the same space, but their real resolutions,
$\hCs{g,n}$ and $\hMs{g,n+1},$ are different -- because the former is
viewed as the domain of a Lefschetz map while the latter is the range (of a
different map); so it is the Lefschetz map itself which is resolved by this
construction.

\begin{proposition}\label{MetResMod.344} After the `metric resolution' of
  domain and range a Lefschetz map $\phi:C\longrightarrow M$ lifts to a b-fibration  
\begin{equation}
\hphi:\hC\longrightarrow \hM.
\label{MetResMod.296}\end{equation}
\end{proposition}

\begin{proof} As noted in Proposition~\ref{MetResMod.343}, after
  the initial real resolution the lift of $\phi$ is smooth and a
  b-fibration. However, after the logarithmic step in
  Definition~\ref{MetResMod.274} regularity (and indeed existence) fails 
  precisely in the new normal variables, the defining functions
  $\ilog\rho.$ Tangential regularity is unaffected and there is no issue
  where the map is a fibration, but from the local form
  \eqref{MetResMod.294} away from the new boundaries,
\begin{equation}
\phi^*\ilog r=\frac{\ilog r_1\ilog r_2}{\ilog r_1+\ilog r_2}
\label{MetResMod.297}\end{equation}
and the right side is not smooth. However, the second blow-up introduces
the radial variable $R=\ilog r_1+\ilog r_2$ as defining function for the new
hypersurface and makes the `angular functions' $w_1=\ilog r_1/R$ and $w_2=\ilog
r_2/R$ smooth local boundary defining functions, so then
\begin{equation}
\hphi^*\ilog r=Rw_1w_2
\label{MetResMod.298}\end{equation}
is indeed smooth and shows the resulting map to be a b-fibration. Although
the lifted boundary defining function is the product of three boundary
defining functions only two of these can vanish simultaneously.
\end{proof}

After the metric resolution, the domain space $\hC$ has three types of
boundary hypersurfaces. The fixed hypersurfaces denoted $F_*,$ which
correspond to the marked points. These do not have self-intersections nor
do they intersect among themselves. The second class of boundary
hypersurfaces are the lifts (proper transforms) of the original Lefschetz
divisors, we denote them $H_{\I,*}.$ The third class, $H_{\II,*}$ arise
from the final blow up of the codimension two surfaces formed by the
singular points of the Lefschetz map. These last two classes fiber under
$\hphi.$ The fibers of the $H_{I,*}$ consist of circle bundles over the
component Riemann surfaces of singular fibers of $\phi$ with the marked
points blown up in the base of the circle fibrations (forming the
intersections with the $F_*)$ and the nodal points similarly blown up (and
separated) forming the intersections with the $H_{\II,*}.$ Thus these
fibers are all circle bundles over Riemann surfaces with (resolved) cusp
boundaries. The fibers of the $H_{\II,*}$ cylinders each carry torus
bundles, linking the boundary curves of the fibers of $H_{I,*}$
corresponding to the nodes.

The boundaries of the resolved base $\hM$ also carry circle bundles, over
which the boundary faces above fiber. From Proposition~\ref{MetResMod.411},
the normal bundle of any $G_{i}'$ is trivial, hence there is a uniquely
defined circle bundle. The faces $H_{\II,*}$ fiber over the original
intersection of the divisors in $M$ as a torus bundle over a cylinder, with
a circle subbundle corresponding to the diagonal action in the angular
variables in \eqref{MetResMod.294}.

\begin{proposition}\label{BunDef} The circle bundles, and torus bundles in
  the case of $H_{\II}\subset\hC,$ over the boundary hypersurfaces of $\hM$
  and $\hC$ have well-defined extensions off the boundaries up to infinite order
  and the rotation-invariant defining functions are determined up to second
  order.
\end{proposition}

\begin{proof} A complex hypersurface in a complex manifold has a local
  defining function $z$ which is well-defined up to a non-vanishing complex
  multiple so another defining function is 
\begin{equation}
z'=\alpha z(1+z\beta(z))
\label{MetResMod.315}\end{equation}
where $\alpha$ is independent of $z.$ Thus a branch of the logarithm
satisfies 
\begin{equation}
\log z'=\log z+\log\alpha +\log(1+z\beta).
\label{MetResMod.316}\end{equation}
In terms of $\rho =\ilog|z|$ the last term vanishes to infinite order with
$\rho$ so the circle action, given by the imaginary part of
\eqref{MetResMod.316} is defined up to infinite order. The real part of
\eqref{MetResMod.316} shows that the new defining function for the
hypersurface is
$$
\rho'=\ilog|z'|=\frac{\ilog|z|}{1-i\log|z|\log|\alpha|+\log|1+z\beta|}=\rho +O(\rho ^2).
$$
Thus there is a determined class of defining functions differing only by
second order terms. Under the further, radial, blow up the torus bundle and
defining function 
\begin{equation}
\rho _{\II}=\rho _1+\rho _2
\label{MetResMod.348}\end{equation}
have similarly well-defined extensions.
\end{proof}

\begin{definition}\label{MetResMod.381} The space of smooth `rotationally
  flat' functions
on $\hC$ which are invariant to infinite order under the circle and torus
  actions will be denoted $\CI_{\theta}(\hC)\subset\CI(\hC).$
\end{definition}

\noindent There is a similar notion on $\hM$ and for metrics on $\hC$ and
$\hM$ and other bundles to which the circle and torus actions naturally extend.
The discussion above shows that the boundary defining functions
are all rotationally flat, 
\begin{equation}
\rho _F,\ \rho _{\I},\ \rho _{\II}\in\CI_{\theta}(\hC).
\label{MetResMod.382}\end{equation}

Such a defining function, determined to leading order, is particularly
relevant at the boundary hypersurfaces, $H_{\II,*}\subset\hC,$
corresponding to the blow-up of the variety of singular points of the
Lefschetz fibration; we say it determines a cusp structure there and denote
by $\rho _{\II}\in\CI(\hM)$ an admissible defining function in this
sense. As a direct consequence of Proposition~\ref{MetResMod.344}, the null
bundle of the differential $\hphi_*$ extends smoothly to a subbundle of
$\bT\hM.$ The holomorphy of $\phi$ means this has a complex structure away
from the boundaries but this does not extend to a complex structure on the
null bundle, just corresponding to the fact that the null bundle after the
first resolution does not lift to the null bundle of the differential after
the second stage. Rather it lifts to a rescaled version of this,
corresponding to the `cusp' structure at $H_{\II,*}.$

\begin{figure}[htp]
\centering
%



\includegraphics{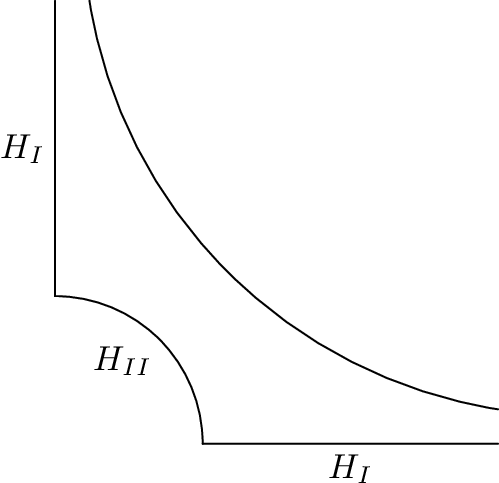}
\caption{The metric resolution that introduces two boundary hypersurfaces $H_{\I}$ and $H_{\II}$}
\end{figure}

\begin{lemma}\label{MetResMod.345} The null bundle of the Lefschetz map
  lifts, with its complex structure, to the rescaled bundle which has
  global sections 
\begin{equation}
\CI(\hC;\hL)=\{v\in\CI(\hC;\bT\hM);\hphi_*(v)=0\Mand v\rho_{\II}=O(\rho_{\II}^2)\}
\label{MetResMod.346}\end{equation}
and is equivariant to infinite order with respect to the circle and torus actions.
\end{lemma}

\begin{proof} The result is immediate away from $H_{\II,*}.$ In plumbing
  coordinates $z=r_1e^{i\theta_1},$ $w=r_2e^{i\theta_2}$ near the singular
  points the null space of the differential of $\phi$ lifts in terms of
  $\rho_{i}=\ilog r_i$ to be spanned by
\begin{equation}
\begin{gathered}
v=r_1\pa_{r_1}-r_2\pa_{r_2}+i(\pa_{\theta_1}-\pa_{\theta_2})=
\rho_1^2\pa_{\rho_1}-\rho_2^2\pa_{\rho_2}+i(\pa_{\theta_1}-\pa_{\theta_2}),\\
v(\rho_1+\rho_2)=\rho_1^2-\rho_2^2=(\rho_1+\rho_2)^2\left(\frac{\rho_1}{\rho_1+\rho_2}-\frac{\rho_2}{\rho_1+\rho_2}\right).
\end{gathered}
\label{MetResMod.347}\end{equation}
Since $\rho_{\II}=\rho_1+\rho_2$ is a local admissible defining function for
$H_{\II,*}\subset\hM$ and the projective functions are smooth on $\hM$ this
shows that the null bundle lifts to \eqref{MetResMod.346} with its complex
structure remaining smooth.
\end{proof}

\section{Log-smoothness of the fiber metrics}\label{outline}

We proceed to outline the proof of Theorem~\ref{MetResMod.342} in the
context of the metric resolution of a general smooth stable Lefschetz
fibration, as defined above, $\hC\longrightarrow\hM.$ Before the real
resolution, the regular fibers of $C\longrightarrow M$ are Riemann
surfaces, with punctures where they intersect the fixed divisors
$F_i\subset C$ which are mapped surjectively to $M.$ These fixed divisors
cannot intersect, and in consequence the regular fibers are naturally
pointed Riemann surfaces. Similarly, the singular fibers are nodal Riemann
surfaces, with punctures at the intersections with the $F_i,$ which are
disjoint from the nodes. The assumed stability implies the stability of
each of the component punctured Riemann surfaces, after each of the nodes
has been separated into a pair of punctures.

Under these assumptions of stability on each fiber there is a unique
Riemannian metric of curvature $-1$ and finite area which is complete
outside the nodes and marked points; these are the hyperbolic fiber
metrics. We extend the main result of \cite{MetLef}, which corresponds to
the formation of a single node, to the general case. The basic properties
of log-smooth functions on a manifold with corners, in this case $\hC,$ are
recalled in the Appendix. In particular, rotational flatness, meaning that
application of a generator of the circle action corresponding to a resolved
divisor in $\hM$ lifted from the base yields an object vanishing to
infinite order at the preimage of that resolved divisor, extends directly
to the log-smooth case. 

\begin{theorem}\label{MetResMod.285} The hyperbolic fiber metrics for a
  stable Lefschetz map form a log-smooth rotationally flat family of
  hermitian metrics on the complex line bundle $\hL$ over $\hC.$
\end{theorem}

\noindent Theorem~\ref{MetResMod.342}, the universal case
$\hCs{g,n}\longrightarrow \hMs{g,n},$ follows from the uniqueness of the
fiber metrics by localizing in the base and passing to a finite cover to
remove the orbifold points.

The proof is carried out in the subsequent sections concluding in
Section~\ref{Conformal} as follows:

\begin{enumerate}
\item In Section~\ref{Grafting} by slightly extending the `grafting'
  construction of Obitsu and Wolpert we construct, in
  Proposition~\ref{MetResMod.383}, a smooth family of 
  Hermitian fiber metrics on $\hL$ which has constant curvature $-1$ near
  the nodal parts and the fixed divisors, is rotational invariant to
  infinite order at the preimage of each divisor in the base and reduces to
  the standard cusp metric on the singular fibers up to quadratic
  error.
\label{PF1}
\item In Section~\ref{Linearized} we examine in detail the properties of
  $(\Lap+2)^{-1}$ for this family of metrics since this is the operator
  appearing in the linearization of the curvature equation in the form
\begin{equation}
(\Lap_{g_{\pl}}+2)u=h=-1-R(g_{pl}).
\label{MetResMod.386}\end{equation}
First we show the uniform boundedness of this operator on appropriate
spaces in which the parameters are incorporated. Next we show, inductively,
the existence of rotationally invariant solutions in the sense of formal
log-power series at the boundary faces lying in the preimages of the
divisors. Combining these two results shows that $(\Lap+2)^{-1}$ acts on
log-smooth rotationally invariant data. \label{PF2}

\item From the `grafted' family of metrics from
  Proposition~\ref{MetResMod.383} below, the regularity of the constant
  curvature family is obtained in Section~\ref{Conformal} by solving the
  equation for the conformal factor $e^{2f}$
\begin{equation}
\Lap_{g_{\pl}}f+e^{2f}+R(g_{pl})=0.
\label{MetResMod.388}\end{equation}
This is first shown to have a formal log-power series solution, by
iteration of the corresponding result for the linearized equation, and then an
application of the implicit function theorem shows that this is the
expansion of the unique solution to \eqref{MetResMod.388} on the regular fibers.
\label{Pf3}
\end{enumerate}

In Section~\ref{WP} this result is applied to yield a corresponding
regularity statement for the Weil-Petersson metric.

\section{The grafted metric}\label{Grafting}

We carry through the construction of a family of metrics on $\hL$ near a
given singular fiber of $\hpsi$ as described in Step (\ref{PF1}) above. As
already noted, this is essentially the construction of Obitsu and
Wolpert~\cite{MR2399166} . Initially we work locally in the base in the
original holomorphic setting, for the universal case passing to a finite
cover of the base to remove orbifold points, but then the metric can be
averaged over the finite group action.

If the base fiber has $k$ pairs of nodal points it lies in the intersection
of $k$ local divisors over which the family is holomorphic. The
uniformization theorem, with smooth parameters, shows the existence of a
smooth family of hyperbolic metrics for this restricted family. Away from
the $2k$ nodal points the family forms a smooth fibration, so with smoothly
varying complex structures and in a smooth local trivialization give a
smooth family of Hermitian metrics.

Near each of the nodal points there are holomorphic coordinates in which
the map is the product of a Lefschetz map and a projection; the $k$-fold
divisor is a smooth submanifold of the zero set of the Lefschetz factor and
the hyperbolic family over this submanifold has cusp singularities along an
intersection of $k-1$ divisors contained in the intersecting pair of
divisors defined by the local Lefschetz singularity. In this sense the
complex structure on the fiber near the nodal points can be arranged to be
locally constant. On each fiber there is a holomorphic coordinate near the
cusp point in terms of which the metric take the standard cusp
form. This fiber holomorphic coordinate can be extended smoothly to yield a
smooth, but not holomorphic, complex defining function for the nodal surface.
This allows the family of constant curvature metrics to be extended off the
singular surface, near the nodal crossing, by the plumbing metric. This
yields a local extension of the initial family of metrics to hyperbolic
metrics which are smooth on $\hL$ near each of the nodal points.

Finally these $2k+1$ families -- the plumbing metric near each of the $2k$
resolved nodes and a smooth family elsewhere -- may be combined to give a
smooth family of rotationally flat metrics on $\hL.$

\begin{proposition}\label{MetResMod.383} Near each point in a $k$-fold
  intersection of divisors in $\MH$ there is a smooth family of
  Hermitian metrics on $\hL$ which restricts to the hyperbolic metrics over
  the intersection of divisors, is equal to the plumbing metric near
  the resolved nodes $H_{\II,*}$ and to the standard metric near the fixed divisors, is
  rotationally flat and has curvature equal to $-1$ to second order at the divisors in $H_{\I, *}.$
\end{proposition}

\begin{proof} The various families discussed above are Hermitian, for the
  same bundle and are rotationally flat. The different local components of
  $H_{\II,*}$ do not intersect so the discussion reduces to the case of a
  single Lefschetz factor in $\CH,$ patching the extension of the
  limiting metric to the grafting metric. Since they differ only by
  quadratic terms in $\rho _{\I}$ away from (but near) $H_{\II,*}$ it is
  only necessary to use a rotationally flat partition of unity to combine
  the conformal factors.
\end{proof}

\section{The linearized operator}\label{Linearized}

The boundary of the metric resolution of the total space, $\hC,$ has in
general, three types of boundary hypersurfaces. There are the `fixed'
hypersurfaces, $F_{*},$ arising from the resolution of marked points, the
collective hypersurfaces $H_{\I,*}$ which are fibered by the resolved
Riemann surface and the `necks', $H_{\II,*},$ coming from the resolution of
the nodal surfaces. Defining functions for these hypersurfaces will be
denoted $\rho _{\cF},$ $\rho_{\I}$ and $\rho _{\II}$ with the same notation
used for the restrictions of these functions to the other hypersurfaces,
where they are again boundary defining functions. Note that the collective
hypersurfaces $\cF$ and $H_{\II}$ (each consisting of embedded boundary
hypersurfaces intersecting to produce corners) do not meet each other, so
each only meets the collective hypersurface $H_{\I}.$ It should be noted
that, the defining function $\rho_{\I,i}$ can be taken to be the preimage
of a defining function for the corresponding resolved divisor in the base,
except near the component, $H_{\II,i},$ of $H_{\II}$ corresponding to the related node,
where $\rho_{\I,i}\rho_{\II,i}$ is a multiple of the defining function from
the base.

The grafted metric is equal to the plumbing metric near the intersection of
the boundary faces $H_{\I}$ and $H_{\II}$ and so, in terms of the local
coordinates introduced in Section~\ref{Real-resolution}, takes the form
\begin{equation}
g=\frac{\pi^2s^2}{\sin^2(\frac{\pi s}{s_w})}
\left(\frac{ds_w^2}{s_w^4}+d\theta_w^2\right) 
=\frac{\pi^2s^2}{\sin^2(\frac{\pi}{1+\rho_z})}
\left(\frac{d\rho_z^2}{s^2(1+\rho_z)^4}+d\theta_z^2\right).
\label{MetResMod.30}\end{equation}
Here $s=\ilog|t|, s_{w}=\ilog|w|, \rho_{z}=\ilog|z|/\ilog|w|$.

The Laplacian of the metric near the boundary of $H_{\I}$ is
\begin{equation}
\Lap_{\I}=-\frac{\sin^{2}(\frac{\pi s}{s_{w}})}{(\frac{\pi s}{s_{w}})^{2}}
\left((s_{w}\partial_{s_{w}})^{2}+s_{w}\pa_{s_{w}} + \frac{1}{s_{w}^{2}}\pa_{\theta_{w}}\right),
\end{equation}
where $s_{w}=\rho_{\II}$ is the boundary defining function of $H_{\II}$ and
hence a variable restricted on $H_{\I}$. Similarly, near $H_{\II}$ the Laplacian is 
\begin{equation}
\Lap_{\II}=-\frac{\sin^{2}(\frac{\pi}{1+\rho_{z}})}{(\frac{\pi}{1+\rho_{z}})^{2}}
\left( (1+\rho_{z})^{2}\pa_{\rho_{z}}^{2} 
+ 2(1+\rho_{z})\pa_{\rho_{z}} +\frac{\pa_{\theta}^{2}}{s^{2}(1+\rho_{z})^{2}}
\right)
\end{equation}
with
$\rho_{z}=\rho_{\I}$ is the boundary defining function of $H_{\I}$ and a variable on $H_{\II}$.

The plumbing metric is rotationally invariant to infinite order near all
boundary hypersurfaces. Let $\CI_{\theta}(\CH)\subset\CI(\CH)$ denote the
subspace annihilated to infinite order at each boundary hypersurface by the
corresponding generator(s) of the circle bundle(s) $D _{\theta}$, see
Definition~\ref{MetResMod.381}. So to each hypersurface in $H_{\I}$ and
$\cF$ there corresponds one generator and to each component of $H_{\II}$
there correspond two; note that as discussed in Proposition~\ref{BunDef}
these are all defined up to infinite order in view of the introduction of
logarithmic variables.

As indicated in Step (2) in~\S\ref{outline}, we first study the linearized
equation~\eqref{MetResMod.386}:
\begin{equation*}
(\Lap_{g_{\pl}}+2)u=h=-1-R(g_{pl}).
\end{equation*}
Since $-2$ is outside the spectrum of $\Lap$ the resolvent
$\|(\Lap+2)^{-1}\|\le1$ on the regular fibers of $\hC$ has a bound, in
terms of the $L^2$ norm on the fibers, which is independent of the
parameters. In \cite[Proposition~3]{MetLef} the space
$L^2_{\bo}(\hM;L^2(dg))$ of $L^2$ functions, with respect to the b-volume
form on the base, with values in the fiber $L^2$ spaces is identified with
the total weighted $L^2$ space $\rho_{\II}^{-\ha}L^2_{\bo}(\hC).$ This
statement is a tautology except near the nodal hypersurfaces defined by the
$\rho_{\II}.$ Since these are disjoint, the argument for a single node in
\cite{MetLef} carries over unchanged, as does the commutation argument
giving higher regularity with respect to tangential differentiation in all
variables. In consequence the uniform solvability of \eqref{MetResMod.386},
in spaces including the parameters, follows from these same arguments.

\begin{proposition}\label{MetResMod.129}
For the Laplacian of the grafted metric
\begin{equation}
(\Lap+2)^{-1}:\rho _{\II}^{-\ha}H^k_{\bo}(\CH)\longrightarrow \rho
  _{\II}^{-\ha}H^k_{\bo}(\CH)\ \forall\ k\in\bbN.
\label{Lapiso}\end{equation}
\end{proposition}

In particular this result holds for $k=\infty,$ where the spaces become the
$L^2$ based conormal spaces, so conormality on $\CH$ up to the boundaries
for the solution of $(\Lap+2)u=h$ follow from conormality for the forcing
term. Defining functions in the base commute with the operator so
(vanishing) weights with respect to these commute with the inverse.

More refined regularity properties for the family $(\Lap+2)^{-1}$ follow
from an iterative construction of formal series solutions locally near the
preimage of a point in an intersection of divisors in $\hM$, and is based on the
solution to two model problems which  are described explicitly below and we note
their basic solvability properties here.

The first model operator is $\Lap_{\I}+2,$ obtained from $\Lap+2$ by
restriction to the fibers of the hypersurface $H_{\I}$ above the interior
of a particular corner in $\hM,$ i.e.\ intersection of divisors. These
fibers are Riemann surfaces with constant curvature cusp metrics. The main
issue here is the appearance of logarithmic terms, so of order $\log\log
1/|t|$ with respect to the original complex parameters. Thus the solution
we obtain is log-smooth rather than a true formal power series.

\begin{lemma}\label{MetResMod.429} The $L^2$ inverse of $\Lap+2,$ for a cusp
metric on a Riemann surface $S,$ applied to functions
rotationally-invariant at the ends, satisfies
\begin{multline}
h=\sum\limits_{0\le j\le k}\rho_{\II}(\log\rho_{\II})^j h_j,\ h_j\in\CI_{\theta}(S)\Longrightarrow \\
(\Lap+2)^{-1}h=\sum\limits_{0\le j\le k+1}\rho_{\II}(\log\rho_{\II})^j u_j,\ u_j\in\CI_{\theta}(S).
\label{MetResMod.421}\end{multline}
\end{lemma}
\begin{proof} The  proof is the same as in~\cite[Lemma 4]{MetLef}. The Laplacian is essentially
  self-adjoint and non-negative, so $\Lap+2$ is invertible. Near the boundary the
  zero Fourier mode satisfies a reduced, ordinary differential, equation
  with regular singular points and having indicial roots $1$ and $-2$ in
  terms of a defining function for the (resolved) cusps. Then the form of solution~\eqref{MetResMod.421} follows.
\end{proof}

The second model operator is the ordinary differential operator arising
from $\Lap+2$ conjugated by $\rho _{\II}$ and restricted to the `necks'
$H_{\II}$ and then projected onto the the zero Fourier mode. The fibers are
now cylinders, projecting to interval, $I.$ These operator have a regular
singular points with two indicial roots, $0$ and $3,$ where the first
corresponds to the simple vanishing in \eqref{MetResMod.421}.

\begin{lemma}\label{MetResMod.430} For the (reduced) model operator
  $\Lap_{\II}+2$ the Dirichlet problem is uniquely solvable and for smooth
  boundary data and a smooth forcing term has solution 
\begin{multline}
\text{If }v_{\pm}\in\bbR,\ r=\sum\limits_{0\le j\le
  k}(\log\rho_{\I})^jr_j,\ r_j\in\CI(I)\Mthen\\
(\Lap_{\II}+2)v=r,\ v\big|_{\pa I}=v_{\pm}\Longrightarrow \ v=\sum\limits_{0\le j\le
  k}(\log\rho_{\I})^jv'_j+\sum\limits_{0\le j\le k+1}(\log\rho_{\I})^j\rho_{\I}^3v''_j.
\label{MetResMod.431}\end{multline}
\end{lemma}

\begin{proof} This is the same proof as~\cite[Lemma 6]{MetLef}.  Unique
  solvability follows by integration by parts and positivity.  The initial
  source of logarithmic terms is the `second' indicial root for
  $\Lap_{\II}+2,$ even if $r=0$ but the boundary data differ this
  introduces a   logarithmic term with coefficient vanishing to third order
  at the ends. Ultimately this is an effect of scattering on $H_{\I}.$
\end{proof}

In the inductive argument we begin by solving \eqref{MetResMod.421} on the
highest codimension boundary face in the base, where it is uniform, and
then extend the solution to a neighborhood of the preimage. Near the
(fixed) cusps this can be done so that the extension remains in the null
space of $\Lap+2$ since the operator is actually independent of the
parameters. However along the boundary, $H_{\II},$ corresponding to the
resolved nodes this produces an error which is not rapidly vanishing. On
iteration this requires us to solve the same equation, but now with
$h\in\rho\,\CI(S)$ and then the result is that the solution is a sum of two
terms, one in $\rho\,\CI(S)$ and a more singular term in the space
$\rho\log\rho\,\CI(S).$ Repeating this procedure results in higher powers
of logarithms. For this reason we devote considerable effort below to
controlling the growth rate of the powers of the logs of $\rho _{\I}$ and
$\rho _{\II}$ in the formal power series solution.

In \cite{MetLef} the log-smooth expansion of solutions to
\eqref{MetResMod.386} was shown in the simplest case where only one pair of
nodes forms. Following that argument, with the minor modifications due to
the presence of the other cusps, this generalizes directly to the case of a
neighborhood of a point in a single hypersurface in the base. In fact we
extend this a little further by considering a neighborhood of a point of
(possibly) higher codimension in the base but with the coefficients of the
forcing term vanishing to infinite order in all but one of the defining
functions.

Without loss of generality suppose that $s_{1}$ is the distinguished
defining function and denote the others collectively as $s'.$
Similarly, let $\rho_{\I}=\rho _{\I,1}$ and $\rho_{\II}=\rho_{\II,1}$ be defining functions for
$H_{\I,1}$ and $H_{\II,1}$ in $\hC,$ always taken to reduce to $s_{1}$
outside a small neighborhood of the intersection $H_{\I,1}\cap H_{\II,1}.$
Then, as in \cite{MetLef}, to capture the special structure of the
expansion, we introduced spaces
of polynomials in $\log\rho _{\I}$ and $\log\rho _{\II}$ with coefficients
now in $(s')^{\infty}\CI_{\theta}(\CH),$ the space of smooth functions in
$\hC$ vanishing to infinite order at all the hypersurfaces defined by the
$s'$ and also rotationally invariant to infinite order under all the
circle actions:
\begin{gather*}
\cP^{k}=\left\{u=\sum\limits_{0\le l+p\le k}(\log\rho
_{\I})^l(\log\rho _{\II})^pu_{l,p},\ u_{l,p}\in(s')^{\infty}\CI_{\theta}(\CH)\right\}\\
\cP^{k,m}_{\II}=  \left\{u=\sum\limits_{0\le l+p\le k,\ p\le m}(\log\rho
_{\I})^l(\log\rho _{\II})^pu_{l,p},\ u_{l,p}\in(s')^{\infty}\CI_{\theta}(\CH)\right\}, \ m\le k.
\end{gather*}
The second collection give the filtration of the first spaces by the maximal order of
powers of $\log\rho _{II}:$

\begin{proposition}[Compare Proposition 6
    of~\cite{MetLef}]\label{MetResMod.78} For each $k$
\begin{equation}
h\in\rho_{\II}\cP^{k}+\rho _{\I}\rho_{\II}\cP^{k+1}\Longrightarrow
\ \exists\ u\in\rho_{\II}\cP^{k+1}+\rho  _{\I}^2\rho _{\II}\cP^{k+2,k+1}_{\II}
\label{MetLef.61}\end{equation}
such that
\begin{equation}
(\Lap+2)u-h\in s\left(\rho _{\II}\cP^{k+1}+\rho _{\I}\rho _{\II}\cP^{k+2}\right).
\label{MetResMod.73}\end{equation}
\end{proposition}
\noindent The form of the error allows for immediate iteration.

As noted above, this follows by slight extension of the arguments in
\cite{MetLef}, and also directly from the more general result below.

We proceed by induction over the local codimension $m$ in the base. Thus
the boundary hypersurfaces in $\hM$ have defining functions $s_i=s_{t,i},$
$i\in\{1,\dots,m\}.$ We proceed by separating these boundary hypersurfaces
into collections corresponding to choice of a subset
$L\subset\{1,\dots,m\},$ $\#(L)=p,$ $L=\{j_1,\dots,j_p\}.$ First consider
the space of polynomials determined by a given multidegree
$\kappa=(\kappa_{j_{1}}, \dots, \kappa_{j_{p}})\in\bbN_0^p$ with smooth
coefficients independent of the appropriate angular variables at the
boundary hypersurfaces faces and vanishing rapidly at the boundary faces on
which $s_j=0$ for $j\notin L;$ this space will be denoted 
\begin{equation}
s_{\complement L}^\infty\CI_\theta(\CH),\ s_{\complement L}=\left(\prod_{j\notin L}s_j\right).
\label{MetResMod.422}\end{equation}
Then the space of polynomials is
\begin{equation*}
\cP_L^{\kappa}=\{\sum_{\alpha_{\I,j}+\alpha_{\II,j}\le\kappa_j}a_{\alpha}\prod_{j\in L}
(\log\rho_{\I,j})^{\alpha_{\I,j}}(\log \rho_{\II,j})^{\alpha _{\II,j}},
a_{\alpha}\in s_{\complement L}^\infty\CI_\theta(\CH)\}.
\label{MetResMod.385}\end{equation*}
Here $\alpha=(\alpha_{j_1},\dots,\alpha_{j_p}) \in (\bbN_{0}^{2})^{p}$ with
$\alpha_{j}=(\alpha _{\I,j},\alpha_{\II,j})\in\bbN_0^2, j\in L$ are the
indices for the powers of logarithms in the two boundary defining functions
$\rho_{\I,j}$ and $\rho_{\II,j}$ near the hypersurface $H_{\II,j}$ such
that $s_j=\rho _{\I,j}\rho _{\II,j}.$ This gives a well-defined space of
functions near the union of the boundary faces determined by $L.$ We also
make use of subspaces which have relatively lower order in the second set
of variables
\begin{equation*}
\cP^{\kappa}_{L,\II}=
\{\sum\limits_{\substack{\alpha_{\I,j}+\alpha_{\II,j}\le\kappa_j\\
\alpha_{\II,j}\le \kappa_j-1}}a_{\alpha}\prod_{j\in L}
(\log\rho_{\I,j})^{\alpha_{\I,j}}(\log \rho_{\II,j})^{\alpha _{\II,j}};
a_{\alpha}\in s_{\complement L}^\infty\CI_\theta(\CH)\}
\end{equation*}
Let $e_{j}$ be the multi-index $(0,...,0,1, 0, ..., 0)$ where the
$j$-th entry is 1, and set
$$
\kappa+ne_{j}=(\kappa_{1}, \dots, \kappa_{j}+n, \dots, \kappa_{p}) \in \bbN_{0}^{p}.
$$
Then consider the polynomials with an extra restriction on one index $j$ that
$\log\rho_{\II,j}$ cannot reach the top degree:
\begin{multline*}
\cP^{\kappa+2e_{j}}_{L,\II}=\\
\{\sum\limits_{\substack{\alpha_{\I,i}+\alpha_{\II,i}\le\kappa_{i}+2\delta_{ij}e_{j}\\ \alpha_{\II,j}\le \kappa_{j}}}a_{\alpha}\prod_{j\in L}
(\log\rho_{\I,j})^{\alpha_{\I,j}}(\log \rho_{\II,j})^{\alpha _{\II,j}};
a_{\alpha}\in s_{\complement L}^\infty\CI_\theta(\CH)\}.
\end{multline*}

Let $\phi\in \CIc(\bbR)$ be a cut-off function with support near $x=0$. Since
$\rho_{\II,i}$ never vanishes simultaneously with any other $\rho_{\II,j}$,
if we take $\phi$ with sufficiently small support $\phi (\rho_{\II,i})$ is
a function that is 1 near $\rho_{\II,i}=0$ and vanishes near
$\rho_{\II,j}=0$ for $j \ne i$. We also consider a collective boundary defining
function for the $H_{\II,i}$
$$
\rho_{\II}^{L} = \prod\limits_{j\in L} \rho_{\II,j}.
$$

\begin{proposition}\label{MetResMod.98} Suppose $L\subset\{1,2,\dots,m\}$ and
  any $\kappa\in\bbN_{0}^{p}$ then for any
\begin{equation}
h \in \cL_{L}^{\kappa}=
\rho_{\II}^{L} \cP^{\kappa}_{L}+
\rho_{\II}^{L} \sum_{j \in L} \rho_{\I,j} \phi(\rho_{\II,j})\cP^{\kappa}_{L,\II},
\label{MetResMod.424}\end{equation}
there exists 
\begin{equation}
u\in\cU_{L}^{\kappa}
=\rho_{\II}^{L}\sum_{i\in L}\cP_{L,\II}^{\kappa+e_{i}}
+\rho_{\II}^{L} \sum_{j\in L}\rho_{\I,j}^2 \phi(\rho_{\II,j})  \cP^{\kappa+2e_{j}}_{L,\II}
\label{MetResMod.423}\end{equation}
such that 
\begin{equation}
(\Lap+2)u-h=\sum_{j \in L}s_{j} g_j,\ g_j\in\cL_{L}^{\kappa+1}.
\label{MetResMod.387}\end{equation}
\end{proposition}

\begin{proof} We first solve the equation on the most singular fiber $S=\{\rho _{\I,j}=0, \forall j \in L\}$ which is a punctured Riemann surface with cusp ends (from both the fixed divisors and the nodes), where the function $h$ has
  a finite expansion  
\begin{equation}
h=\sum\limits_{\alpha_{j}\leq \kappa_{j}} a_{\alpha}\prod_{j\in L}(\log\rho _{\I,j})^{\alpha_{j}}
\label{MetResMod.90}\end{equation}
with coefficients $a_{\alpha}$ which are polynomials in $(\log\rho
_{\II,j})$'s of multi-index orders at most $(\kappa_{j}-\alpha_{j})$'s
respectively. For fixed $(\alpha_{\I,1}, \dots, \alpha_{ \I,p})$ the
equation
\begin{equation}
(\Lap+2)\sum\limits_{\alpha_{j} \leq \kappa_{j}} v_{\alpha}\prod_{j\in L} (\log\rho _{\I,j})^{\alpha_{j}}=
h+F,
\label{MetResMod.91}\end{equation}
where $F$ is to vanish at $S,$ induces equation on the coefficients restricted
to $S.$ These form an upper-triangular matrix of b-differential operators,
with diagonal entries $\Lap_S+2,$ so can be solved iteratively, over
decreasing $\sum_{i\in L} \alpha_{\I,i},$ and at each level take the form
\begin{equation}
(\Lap_S+2)v_{\alpha}=g_{\alpha}=h_{\alpha}+\sum\limits_{|\beta|>|\alpha|}P_{\beta}v_{\beta},
\label{MetResMod.92}\end{equation}
where $P_{\beta}$ is the operator that extracts $\prod (\log\rho
_{\I,i})^{\alpha_{\I,i}}$ terms from $v_{\beta}$. It therefore follows
inductively that $g_{\alpha}\in\cP_{L}^{\kappa-\alpha}$ and that 
$v_{\alpha}\in\cP_{L}^{\kappa-\alpha+1}.$ It is important to recall here that $\rho
_{\II,i}$ and $\rho _{\II,i'}$ cannot vanish simultaneously when $i\neq i'$. We proceed to
choose an extension of $v_{\alpha}$ off $S,$ but do this separately near
$\rho _{\II,i}=0$  for each $i$. Away from these hypersurfaces any
smooth extension will suffice.

\begin{figure}[htp]
\centering
%
%
%
%
%
\includegraphics{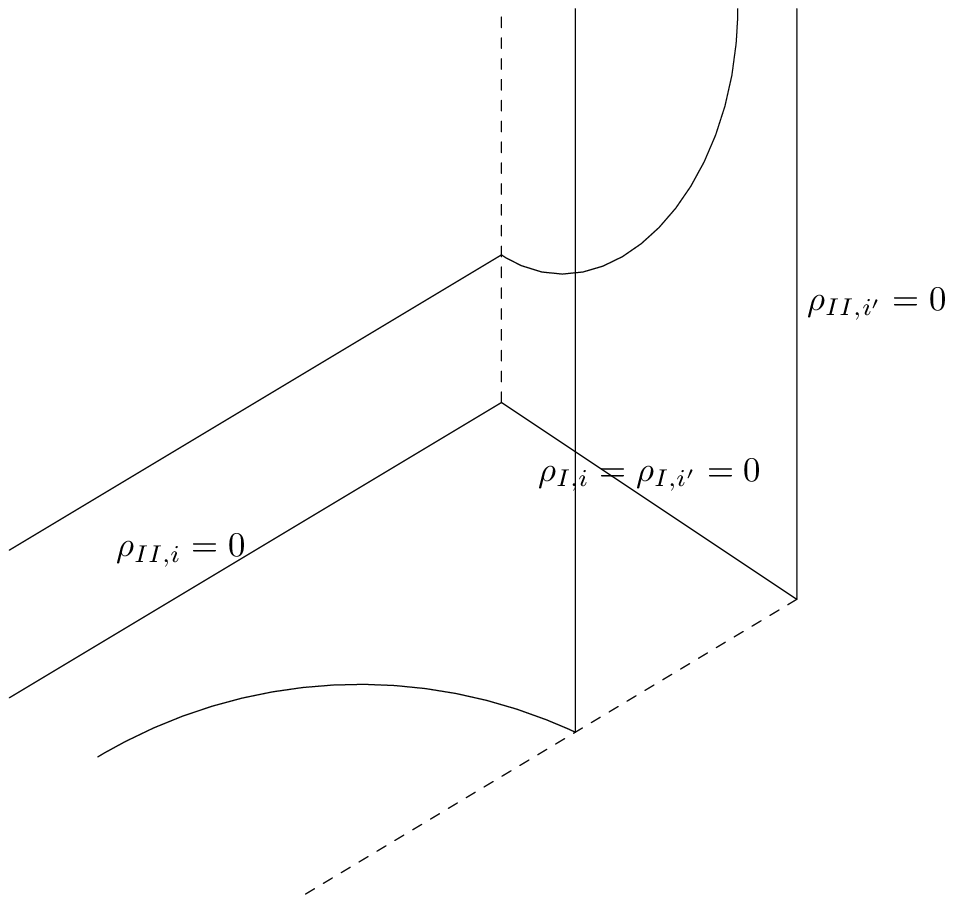}
\caption{The induction procedure with two nodes: curves with two different degenerating directions approaching 
the nodal curve with two nodes}
\end{figure}

Consider the extension of the $v_{\alpha}$ off $S$ near $\rho_{\II,i}=0.$
Since $\rho _{\II,i'}>0$ for any $i'\neq i$, we may replace $\rho_{\I,i'}$ by $s_{i'}$ in this
region. The coefficients $h_{\alpha}$ change, but the form is preserved and
the solutions $v_{\alpha}$ are similarly transformed and the resulting
system is diagonal in $i'.$ Replacing $f$ by $\chi(s_{i})f$ where $\chi$
is a smooth cutoff near zero introduces an error which is in
$s_{i}\cL^{\kappa}_{L}$ and so can be absorbed in the next iteration. Now we can extend the $v_{\alpha}$ from $S$ to
$\rho_{\I,i}=0$ near $\rho_{\II,i}=0$ as solutions of the same equations
although the Laplacian now may depend on $s_{i}$ as a parameter. The error
term vanishes identically at $s_{i}=0$ near $\rho_{\II,i}=0.$ We may now
proceed exactly as in \cite{MetLef} to extend the coefficients $v_{\alpha}$
along $H_{\II,i}$ in such a way as to remove the leading coefficient in
$\rho_{\II,i}.$ Since this is all now uniform in $s_i,$ and the same is
true for the removal of the second term in \eqref{MetResMod.424}, which only
involves the induced differential equation on $H_{\II,i}$, this part of the
solution is certainly of the form \eqref{MetResMod.423} and an error as in
\eqref{MetResMod.387}. A similar construction near $\rho_{\II,i'}=0$ for all other $i'$
completes the proof.
\end{proof}

The results above allow the log-smoothness of the solution of the
linearized equation to be deduced from appropriate log-smoothness of the
data in terms of the spaces defined in \eqref{MetResMod.424} and
\eqref{MetResMod.423} where we drop the suffix $L$ when all boundary
hypersurfaces are involved. Here we abbreviate the indices and write 
$$
s^{\alpha}=\prod_{i=1}^{m} s_{i}^{\alpha_{i}}, \ \alpha=(\alpha_{1}, \dots, \alpha_{m})\in \bbN_{0}^{m}.
$$

\begin{proposition}\label{MetResMod.94} If $h\in\rho _{\II}\CI_{\log}(\CH)$
has expansions at each of the of boundary faces of the form 
\begin{equation}
h\simeq\sum_{\alpha \in \bbN_{0}^{m}} f_\alpha s^\alpha,\ f_{\alpha }\in \cL^{\alpha }
\label{MetResMod.95}\end{equation}
then $u=(\Lap+2)^{-1}h\in\CI_{\log}(\hC)$ has similar expansion
\begin{equation}
u\simeq\sum\limits_{\alpha }u_{\alpha }s^\alpha ,\ u_{\alpha }\in\cU^{\alpha }.
\label{MetResMod.96}\end{equation}
\end{proposition}

\begin{proof} We first recall (see also the Appendix) the structure of the proof of
  Borel's Lemma, summing formal power series, in this context. Suppose
  $x_i\ge0,$ $i=1,\dots,\ell$ and $y$ are respectively boundary defining functions
  and tangential coordinates near a boundary point of codimension $\ell$ on a
  manifold with corners and  
\begin{equation}
\sum\limits_{\substack{\alpha \in\bbN_0^\ell,\beta \in\bbN_0^\ell \\
\alpha _i<\beta _i}}f_{\alpha,\beta }(y)(\log x)^\alpha x^\beta 
\label{MetResMod.424a}\end{equation}
is a formal power series with smooth coefficients of fixed compact
support. Then, choosing a cutoff function $\chi\in\CIc(\bbR^\ell)$ which is
identically equal to $1$ near zero, the series 
\begin{equation}
u=\sum\limits_{\substack{\alpha \in\bbN_0^\ell,\beta \in\bbN_0^\ell\\
\alpha _i<\beta_i}}f_{\alpha,\beta }(y)(\log x)^\alpha x^\beta\chi(x/\epsilon_{\alpha})\in\CI_{\log}
\label{MetResMod.425}\end{equation}
converges provided the $\epsilon _{\alpha}$ decrease rapidly enough. In
fact for an appropriate choice, base on the norms of the coefficients, all
the finite differences 
\begin{equation}
u(x,y)-\sum\limits_{|\alpha|\le N}f_{\alpha,\beta }(y)(\log x)^\alpha
x^\beta 
\label{MetResMod.424b}\end{equation}
are conormal and vanish to order $N$ at the corner. The same arguments
apply when the coefficients are assumed to be on a manifold with corners
but to vanish to infinite order at all boundaries.

The `asymptotic sum' so obtained is independent of the choices up to an
error in $\CI_{\log}$ which vanishes to infinite order at the corner. Such
an error can be decomposed, using a partition of unity in polar coordinates around
(i.e.\ after blowing up) the corner, and such that each term is
non-zero only near one of the $\ell$ faces of codimension $\ell-1,$ i.e. $\{x_{i}=0\}$.  This allows
for the iterative construction of a log-smooth function with expansions at
each of the boundary hypersurfaces, provided these expansions are formally
compatible in essentially the same sense as for Taylor series.

So, to prove that the solution $(\Lap+2)^{-1}h$ is log-smooth we use
Proposition~\ref{MetResMod.98} repeatedly, first on the highest codimension
boundary face in the base $\hM$ so with $\complement L=\emptyset$ and with
a smooth forcing term $h$ with support disjoint from $H_{\II}$ and the
fixed hypersurfaces in $\hC.$ The formal solution obtained from
iteration can be summed to give
\begin{equation}
u_m=\sum\limits_{\alpha}s^\alpha v_\alpha\chi(s/\epsilon _\alpha )\in\rho _{\II}\CI_{\log}(\hC).
\label{MetResMod.427}\end{equation}
Note that although the preimage of $s=0$ in $\hC$ has points of codimension $m+1$ the
series is constructed uniformly at the union of the $m+1$ faces of
codimension $m$ and the same asymptotic summation principles apply
to \eqref{MetResMod.427}. Moreover the error term 
\begin{equation}
h_{m-1}=h-(\Lap+2)u_m
\label{MetResMod.428}\end{equation}
is log-smooth, has expansion with powers coming only from those in the
errors in \eqref{MetResMod.387} at \emph{all} boundary hypersurfaces and
vanishes to infinite order at the preimage of $s=0.$ Thus, using a polar
partition of unity in $s$ it may be divided into $m$ pieces, each of which
vanishes to infinite order with one of the $s_i.$

This allows Proposition~\ref{MetResMod.98} to be applied, now with
$\#(L)=m-1$ and the general case to be proved by induction.
\end{proof}

\section{The curvature equation}\label{Conformal}

In the previous section we showed that the solution to the linearized
equation~\eqref{MetResMod.386} with a log-smooth forcing term is log-smooth. Now we
iteratively apply Proposition~\ref{MetResMod.94} to the full curvature
equation to arrive at the same conclusion for the solution
to~\eqref{MetResMod.388}. Note that the log-smooth functions form a ring.

\begin{proposition}\label{MetResMod.130}
The curvature equation \eqref{MetResMod.388} has a solution in log-power series
\begin{equation}
f\sim\sum\limits_{\alpha }u_{\alpha }s^\alpha ,\ u_{\alpha }\in\cU^{\alpha }
\label{MetResMod.130a}\end{equation}
where the $\cU^\alpha$ are defined by \eqref{MetResMod.423}.
\end{proposition}

\begin{proof}
To get the expansion for the nonlinear equation, we write it as
\begin{equation}\label{MetResMod.354}
(\Lap+2)f=-R-1-(e^{2f}-1-2f).
\end{equation}

Consider the formal solution of the linearized equation $(\Lap+2)f_{1}\sim
-R-1.$ from~\eqref{MetResMod.96}  
$$
f_{1}=\sum\limits_{\alpha \geq 2} s^{\alpha}f_{1,\alpha},\ f_{1,\alpha }\in\cU^\alpha .
$$
Now we solve the nonlinear equation~\eqref{MetResMod.354} by taking
\begin{equation}
f=\sum_{i=1}^{\infty} f_{i}, \ f_{i}=\sum\limits_{\alpha\geq 2i} s^{\alpha}f_{i,\alpha},
\end{equation}
which gives
$$
\sum_{i=1}^{\infty} (\Lap+2) f_{i}=-R-1 - \sum_{k=2}^{\infty} \frac{2^{k}}{k!}(\sum_{i=1}^{\infty}f_{i})^{k}.
$$
Now we require that for $j \geq 2$
$$
(\Lap+2)f_{j}=-\sum_{k=2}^{\infty}\frac{2^{k}}{k!}\left[(\sum_{i=1}^{j-1}f_{i})^{k}- (\sum_{i=1}^{j-2}f_{i})^{k} \right]
=f_{j-1}Q_{j}(f_{1}, \dots, f_{j-1})
$$
where $Q_{j}$ is a polynomial with no constant term. Assuming inductively that
\begin{equation}
f_{k}=\sum_{\alpha\geq 2k} s^{\alpha} f_{k,\alpha}, \ f_{k,\alpha} \in \mathcal{U}^{\alpha}
\end{equation} 
for $k\leq j-1$ as in~\eqref{MetResMod.130a}, the right hand side is in
$\sum_{\alpha\geq 2j} s^{\alpha}\mathcal{L}^{\alpha}.$ Applying
Proposition~\ref{MetResMod.94}, we obtain the inductive conclusion that  
$$
f_{j}=\sum_{\alpha\geq 2j} s^{\alpha} f_{j,\alpha}, \  f_{j,\alpha}\in \mathcal{U}^{\alpha}.
$$ 
This shows the existence of a formal solution as in \eqref{MetResMod.130a}.
\end{proof}

Now the structure of the actual solution follows from an application of the
Implicit Function Theorem.

\begin{proof}[Proof of Theorem~\ref{MetResMod.285}] It suffices to show that there
  is a function $f \in \CI_{\log}(\CH)$ satisfying the curvature equation
  \eqref{MetResMod.388} since then $e^{2f}g_{pl}$ is the family of
  hyperbolic metrics.

Writing $f=f_{0}+\tilde f$, where $f_{0}$ is a sum of the formal solution
obtained from Proposition~\ref{MetResMod.130} by Borel's lemma. Then
$\tilde f$ should satisfy
$$
(\Lap+2)\tilde f =-\left(2\tilde f(e^{2f_0}-1)+e^{2f_0}(e^{2\tilde
  f}-1-2\tilde f)-g\right). 
$$
We can apply Proposition~\ref{MetResMod.129} to see the solvability of this
equation in the space $s^NH^{N}_{\bo}(\hC)$ restricted to a region where at
least one of the $s_i$ is small. The uniqueness of the hyperbolic metric
shows that the $\tilde f$ is in fact smooth and vanishes to infinite order
with each of the $s_i,$ essentially as in~\cite[Proposition 8]{MetLef}.
\end{proof}

\section{The Weil--Petersson metric}\label{WP}

The Weil--Petersson metric is most easily realized at $m\in\Ms{g,n}$ in
terms of the dual metric, an Hermitian metric on the logarithmic cotangent
bundle, $\DL_m^{(1,0)}.$ The definition uses the identification
\eqref{MetResMod.299}. The bundle of quadratic logarithmic
differentials on the resolution of the marked Riemann surface $Z=Z_m$
representing $m\in\Ms{g,n}$ satisfies
$$
(\DL^{(1,0)}Z)^2\otimes(\overline{\DL^{(1,0)}Z})^2=(\DL^{(1,1)}Z)^2.
$$
The hyperbolic metric on $Z,$ complete outside the marked points and
fixed uniquely by the complex structure, has area form given in terms of a
holomorphic coordinate vanishing at a marked point 
\begin{equation}
\frac{dzd\bar z}{|z|^2(\log|z|)^2}.
\label{MetResMod.356}\end{equation}

By definition, the vector space $QZ$ of holomorphic quadratic differentials
on a punctured Riemann surface consists of those holomorphic sections of
$(\DL^{(1,0)}Z)^2$ which vanish at the marked points. Dividing the product
of one such form and the complex conjugate of another by the area form
therefore gives a continuous section of $\DL^{(1,1)}Z$ which is smooth away
from the marked points near which it has a bound
\begin{equation}
\big|\frac{q_1\bar {q}_2}{\mu_m}\big|\le C(\log|z|)^2{dzd\bar
  z},\ q_1,\ q_2\in QZ,
\label{MetResMod.357}\end{equation}
and so is integrable. This integral gives the dual Weil--Petersson metric, as an
Hermitian form, by push-forward under the Lefschetz map
\begin{equation}
\begin{gathered}
\DL^{(1,0)}_m\Ms{g,n}=QZ_m\\
\DL^{(1,0)}_m\Ms{g,n}\otimes\DL^{(0,1)}_m\Ms{g,n}\ni(q_1,\overline{q_2})\longmapsto
G_{\WP}(\zeta_1,\zeta_2)=\phi_*\left(\frac{q_1\bar{q}_2}{\mu_{m}}\right).
\end{gathered}
\label{MetResMod.358}\end{equation}
Standard proofs of the Riemann mapping theorem for marked Riemann surfaces
show the smoothness (indeed real-analyticity) of $\mu_m$ on the moduli
space and hence that $G_{\WP}$ is similarly smooth as a metric on $\Ms{g,n}.$

We proceed to analyze the behavior of the Weil--Petersson (co-)metric near a
boundary point $\bar m\in\oMs{g,n}\setminus\Ms{g,n},$ so in some
intersection of the $G'_{i},$ corresponding to the number of geodesics
which have been pinched to nodes. The fiber of $\opsi$ above $\bar m$ is a
nodal Riemann surfaces lying in one or more of the $G_i.$ The null bundle
$\oL_{\opsi},$ of the log differential of $\opsi$ again reduces to the log
tangent bundle of the nodal surface, in effect separating nodes to
punctures. However the corresponding space of holomorphic quadratic
differentials, $\oQ_{\bar m},$ which is naturally isomorphic to the
logarithmic cotangent space $\DL_{\bar m}\oMs{g,n}$ consists of those
holomorphic sections of the square of the dual bundle $(\oL_{\opsi})^{-2}$
which vanish at the marked points but at nodal points, in one of the $S_i,$
are simply required to take a consistent value, the same at the two
punctures representing the node. Thus in place of \eqref{MetResMod.357},
which again holds at the marked points, only the much weaker bound
\begin{equation}
\big|\frac{q_1\bar {q}_2}{\mu_{\bar m}}\big|\le C(\log|z|)^2\frac{dzd\bar
  z}{|z|^2},\ q_1,\ q_2\in \oQ_{\bar m}
\label{MetResMod.367}\end{equation}
generally holds near the nodes. This does not imply integrability so the
co-metric is singular at $\bar m.$

The structure of the area form at the singular fibers has been analysed in
detail, above, by passing to the real resolution
$\hpsi:\hCs{g,n}\longrightarrow \hMs{g,n};$ see
Theorem~\ref{MetResMod.342}. The compactified moduli space is resolved, to
a tied manifold, with boundary hypersurfaces replacing the exceptional
divisors with their logarithmic real blow up, as for the fixed divisors
(which are in the universal curve). Thus each boundary hypersurface has a
local defining function of the form $s_i=\ilog|z_i|$ and forms a circle
bundle over the corresponding divisor given locally by $z_i=0.$

\begin{lemma}\label{MetResMod.368} The lift of the logarithmic tangent
  bundle $\DL\oMs{g,n}$ to $\hMs{g,n}$ is naturally identified with corresponding
  (iterated) cusp bundle with local spanning sections $\frac{ds_i}{s_i^2},$
  $d\theta_i$ and $d w_j$ at any boundary face.
\end{lemma}

\begin{proof} This is just the computation 
\begin{equation}
\frac{dz_i}{z_i}=\frac{d|z_i|}{|z_i|}+id\theta=\frac{ds_i}{s_i^2}+id\theta,\
s_i=\ilog|z_i|=(\log|z_i|^{-1})^{-1}. 
\label{MetResMod.369}\end{equation}
\end{proof}

The preimage in $\hMs{g,n}$ of $\bar m\in\oMs{g,n}$ lying in a $k$-fold
intersection of the divisors, $G'_{i'},$ is a product of $k$ circles
through a boundary point of codimension $k$ in $\hMs{g,n}.$ The preimage in
$\hCs{g,n}$ of the singular fiber $Z_{\bar m}\subset\oCs{g,n}$ above $\bar
m$ is a product over this $k$-torus with a factor which is an `articulated
manifold' (the real analog of a nodal surface) in the sense that it is a
union of compact surfaces with boundaries meeting only at (some) of their
boundary faces $\hZ=Z_{\I}\cup Z_{\II}.$ The component manifolds forming $Z_{\I}$
are resolutions of the Riemann surfaces (including `bubbled off spheres')
into which the original Riemann surface has decomposed under nodal
degeneration. These can (and in the case of spheres must) have boundary
faces formed by the fixed divisors, with collective boundary defining
function $\rho _{\cF};$ the other boundary faces arise from the resolution
of the (separated) nodes which are in common with components in $Z_{\II}.$
The $Z_{\II}$ are all cylinders, joining the two circles forming the
resolution of a node.

The hyperbolic fiber metric of the original Riemann surface induces the
hyperbolic fiber metric on each of the components of $Z_{\I}$ (which are
necessarily stable) which is therefore of the form \eqref{MetResMod.356}
near its boundaries. The fiber metrics degenerate at the $Z_{\II}$ in an
adiabatic fashion; namely the metric approaches the pull back of metric on
the base of the circle fibration with the tangential part vanishing to
second order.

For the resolved universal curve, $\hpsi:\hCs{g,n}\longrightarrow
\hMs{g,n},$ as shown above, is a real Lefschetz map in the sense that it is
a b-fibration, so has surjective b-differential and the defining functions
on the base each lift to be everywhere either locally a defining function
or the product of two
\begin{equation}
\hpsi^*s_{i}=\rho_{\I,i}\rho _{\II,i}.
\label{MetResMod.363}\end{equation}
The collective boundary hypersurfaces $H_{\I}=\bigcup_i\{\rho _{\I,i}=0\}$ and
$H_{\II}=\bigcup_i\{\rho _{\II,i}=0\}$ are without self-intersections
(because in the case of $H_{\I}$ these have been replaced by intersections
with components of $H_{\II}$ through the blow up of the $S_i)$ and are the
unions of the parts of the fibers just described.

\begin{lemma}\label{MetResMod.364} If $q_1,$ $q_2$ are families of
  holomorphic logarithmic quadratic differentials on $\hCs{g,n},$ near a
  boundary fiber, then if $\nu_{\hpsi}$ is a positive section of the b-fiber
  density bundle for $\hpsi,$
\begin{equation}
\frac{q_1\bar q_2}{\mu_H}=a\rho
_{\II}^{-3}\rho_{\cF}^{\infty}\nu_{\hpsi},\ 0<a\in\CI_{\log}(\hCs{g,n})
\label{MetResMod.365}\end{equation}
\end{lemma}

\begin{proof} This is a refinement of the computation above leading to
  \eqref{MetResMod.357}, \eqref{MetResMod.367}. Near the fixed
  hypersurfaces the metric has a cusp singularity so the fiber area form is  
\begin{equation}
\mu_H=\alpha \frac{d\rho_{\cF}d\theta_{\cF}}{\rho_{\cF}},\ 0<\alpha \in\CI_{\log}.
\label{MetResMod.371}\end{equation}
The product of the holomorphic quadratic differentials vanishes in terms of
the holomorphic coordinates so 
\begin{equation}
q_1\bar q_2=e^{-2/\rho
  _{\cF}}b(\frac{d\rho_{\cF}d\theta_{\cF}}{\rho_{\cF}^2})^2,\ b\in\CI.
\label{MetResMod.372}\end{equation}
The quotient is therefore rapidly decreasing at the fixed boundary
hypersurfaces, giving the formal factor of $\rho_{\cF}^{\infty}$ in
\eqref{MetResMod.365}. 

It remains to analyse the behaviour at $Z_{\II},$ including at the
corresponding boundary faces of $Z_{\I}.$ These cover the nodal points at
which the holomorphic quadratic differentials do not necessarily
vanish. The structure of the area form is essentially the same as in
\eqref{MetResMod.371} with $\rho_{\II}$ replacing $\rho_{\cF}$ but extends
along $H_{\II};$ Lemma~\ref{MetResMod.368} (and the uniform analysis of the
metric) shows that
\begin{equation}
\mu_H=\alpha \frac{d\rho_{\II}d\theta_{\II}}{\rho_{\II}},\ 0<\alpha \in\CI_{\log}.
\label{MetResMod.373}\end{equation}
The exponential factor in \eqref{MetResMod.372} is then missing, so 
\begin{equation}
q_1\bar q_2=b\rho_{\II}^{-4}(d\rho_{\II}d\theta_{\II}))^2,\ b\in\CI.
\label{MetResMod.374}\end{equation}
and \eqref{MetResMod.365} follows.
\end{proof}

The holomorphic quadratic differentials which vanish at a nodal point (so in
some $S_i)$ correspond to the `tangential' (logarithmic) differentials on
$\oMs{g,n}$ at that point, those which are smooth up to the divisor
(although possibly singular along it as logarithmic differentials). If
either $q_1$ or $q_2$ lies in this subspace then at least one exponentially
vanishing factor occurs and \eqref{MetResMod.365} is
replaced by
\begin{equation}
\frac{q_1\bar q_2}{\mu_H}=a\rho
_{\II,i}^{\infty}\rho_{\cF}^{\infty}\nu_{\hpsi},\ 0<a\in\CI_{\log}(\hCs{g,n})\Mnear H_{\II,i},
\label{MetResMod.366}\end{equation}
the corresponding component of $Z_{\II}.$ In this case the area form is locally integrable
whereas in the `normal' case, \eqref{MetResMod.365} it is singular at $Z_{\II}.$

\begin{proof} [Proof of Theorem~\ref{MetResMod.375}] As noted in
  \eqref{MetResMod.358}, the Weil--Petersson metric, defined through the dual
  metric on $\DL^{(1,0)}M,$ is given by push forward under the Lefschetz
  map $\psi.$ Lifting under the metric resolution, it follows that 
\begin{equation}
G_{\WP}(\zeta _1,\zeta _2)=\hpsi_*(\frac{q_1\bar q_2}{\mu_m})
\label{MetResMod.376}\end{equation}
where the $q_i$ are holomorphic quadratic differentials representing the
$\zeta _i.$ Here we may think of the $\zeta_i$ as holomorphic sections of
$\DL^{(1,0)}\oMs{g,n}$ near some point $\bar m$ and the $q_i$ as the
corresponding sections of $\oQ\oMs{g,n},$ hence holomorphic quadratic
differentials near the fiber above $\bar m.$ Since $\hpsi$ is a b-fibration
the push-forward theorem in \cite{MR93i:58148} applies. In principal this
is in the context of manifolds with corners, rather than the slightly more
general case of tied manifolds with orbifold points as encountered
here. However the fibers \emph{are} globally manifolds with boundaries and
the result is essentially local in the base. The same remark applies to the
presence of orbifold points since these are also in the base directions, so
one can always apply the result in \cite{MR93i:58148} to an appropriate
finite local cover and then take the quotient.

\begin{proposition}[See \cite{MR93i:58148}]\label{MetResMod.377} If
  $\hphi:\widehat{C}\longrightarrow \widehat{M}$ is a b-fibration with
  compact fibers, between manifolds with corners, with multiplicity at most
  two, in the sense that each boundary defining function in the range is locally the
  product of at most two boundary defining functions in the domain and
  $\nu_{\hphi}$ is a non-vanishing fibre b-density then
\begin{equation}
\CI_{\log}(\widehat{C})\ni a\longmapsto
\hphi_*(a\nu_{\phi})=g+\sum\limits_{i'}a_{i'}\log\rho _{i'},\ g,\ a_{i'}\in
\CI_{\log}(\widehat{M}).
\label{MetResMod.378}\end{equation}
If $a$ vanishes to order $j$ at the codimension two faces occurring
as the common zero surface in \eqref{MetResMod.363} for $\hphi^*\rho_{i'}$
then the corresponding coefficient $a_{i'}$ vanishes to order $j$ at $\rho_{i'}=0.$ 
\end{proposition}
\noindent
In brief the logarithmic multiplicity in the generalized Taylor series at
boundary faces in the range of such a push-forward is at most one degree
higher, there is at most one more factor of $\log\rho _{i'},$ and this only
arises from the Taylor series at the codimension two faces mapping onto a
given boundary hypersurface. One can be more precise about the sense in
which it is the diagonal terms in the Taylor series at the corners which
contribute to the logarithmic coefficients.

Near a point in the local intersection of $k$ exceptional divisors in the
base, $\oMs{g,n},$ we may always choose a local coordinate basis of $\DL^{(1,0)},$
$\zeta_{i'}=dz_{i'}/z_{i'},$ $i'=1,\dots,k,$ $\zeta'_j=dw_j,$ over an open neighborhood
$O\subset\oMs{g,n}$ so that the corresponding holomorphic quadratic
differentials are $q_{i'}$ and $q'_j,$ where each $q_{i'}$ vanishes at all $S_{l'}$ with
$l'\not=i'$ and takes the value $1$ at $S_{i'}$ and the $q'_j$ are
tangential in the sense that they vanish all local $S_{l'}.$ Thus the
$q'_{j}$ are holomorphic quadratic differentials when the nodal points are
separated and regarded as marked points on the resulting possibly
non-connected Riemann surface.

Applying Proposition~\ref{MetResMod.377} to compute the coefficients of the
metric via \eqref{MetResMod.358} using Lemma~\ref{MetResMod.364} and the
subsequent remark we conclude that locally
\begin{equation}
\begin{gathered}
G_{\WP}(\frac{dz_{i'}}{z_{i'}},\frac{dz_{i'}}{z_{i'}})\in \rho _{i'}^{-3}\CI(\widehat O)+\log\rho _{i'}\CI_{\log}(\widehat O),\\
G_{\WP}(\frac{dz_{i'}}{z_{i'}},\frac{dz_{j'}}{z_{j'}})\in \CI_{\log}(\widehat O),\ i'\not=j'\\
G_{\WP}(dw_j,dw_k)\in \CI_{\log}(\widehat O)\ \forall\ j,k
\end{gathered}
\label{MetResMod.379}\end{equation}
where $\widehat{O}$ is the preimage of $O$ in $\hMs{g,n}.$ 

The singular coefficients in \eqref{MetResMod.379} arise only from the
boundary hypersurface $H_{\II,i'}$ resolving $S_{i'}$ in $\hCs{g,n}.$
Consider the leading terms in the length, with respect to the
Weil--Petersson co-metric, of $dz_{i'}/z_{i'},$ dropping the
index for notational simplicity. This is given by the push-forward formula.
Since we have shown that they differ by quadratic terms at the divisors, it
suffices to replace the Weil--Petersson metric by the grafting metric in the
computation of the leading terms. The explicit computation of the integral
depends on the fibration being in model Lefschetz form and the metric
reducing to the plumbing metric near the nodal points. To accomplish this
we choose $z$ to be fiber holomorphic but only smooth in the base,
i.e.\ arising from a non-holomorphic defining function for the
corresponding $G'_i.$

Now the integral of the fiber area form which may be written out explicitly
as
\begin{equation}
\int_{|t|<|z|<1}\mu_s
\label{MetResMod.242}\end{equation}
where $\mu_s$ is the quotient of $|dz|^4/|z|^4$ -- the square of a
quadratic differential with a double pole -- and the area form of the
plumbing metric. So,  
\begin{equation}
\mu_s=\frac{|dz|^2}{|z|^2}(\log|z|)^2
\left(\frac{\log|t|}{\pi\log|z|}\sin\frac{\pi\log|z|}{\log|t|}\right)^2.
\label{MetResMod.243}\end{equation}
So, setting $s=-1/\log|t|,$ $r=-1/\log|z|$ we find 
\begin{equation}
\begin{gathered}
\frac{|dz|^2}{|z|^2}=\frac{drd\theta}{r^2},\\
\mu_s=\frac{drd\theta}{r^4}\left(\frac r{\pi s}\sin\frac{\pi s}r\right)^2.
\end{gathered}
\label{MetResMod.244}\end{equation}

The integral then becomes 
\begin{equation}
2\pi\int_s^1\frac{dr}{\pi^2s^2r^2}\sin^2\frac{\pi s}r.
\label{MetResMod.245}\end{equation}
Changing variable to $\tau=s/r,$ so $dr/r^2=-d\tau/s,$ gives 
\begin{equation}
\frac1{\pi s^3}\int_s^1(1-\cos 2\pi\tau)d\tau=
\frac{1}{\pi s^3}(1-s+\frac{\sin{2\pi s}}{2\pi}).
\label{MetResMod.246}\end{equation}
The higher order terms in the coefficients of these diagonal terms arise
from either the difference of the conformal factor for the degenerating
family of metrics, which may give a term in $\log \rho
_{i'}\CI_{\log}(\widehat{O}),$ and the higher order terms in the quadratic
differential which produce a term in $\CI_{\log}(\widehat{O}).$

To obtain the form \eqref{MetResMod.369} of the metric we must change from
the non-holomorphic defining functions for the divisors to holomorphic
ones. As discussed in Proposition~\ref{BunDef} this only changes the real
boundary defining functions $s_j$ by quadratic terms and the given
decomposition of the metric is not altered by such changes.
\end{proof}

\section{Ricci curvature and metric}\label{Ricci}

The Ricci curvature of the Weil--Petersson metric itself defines a K\"ahler
metric on the moduli space; the quasi-isometry class was found by Trapani
\cite{MR1176026} and the leading asymptotics at a divisor by Liu, Sun and
Yau \cite{MR2169586,MR2144543}. Near the intersection of $k$ divisors, written out in terms of the singular coordinate basis 
\begin{equation}\label{MetResMod.412}
\alpha _j=d\log
z_i=d(-s_i^{-1}+i\theta_i),
1\le i\le k;\
\alpha_j=dz_j, 3g-3+n\geq j>k
 \end{equation}
the full asymptotic expansion of the
Weil--Petersson metric in turn yields a full description of the asymptotic
behaviour of the Ricci metric at the exceptional divisors:

\begin{theorem}\label{MetResMod.389} In terms of the coordinates $s_i,$
  $\theta_i$ and $z_l$ near an intersection of exceptional divisors, the Ricci metric
derived from the Weil--Petersson metric is $\theta_i$-invariant to infinite
order at $s_i=0,$ has log-smooth coefficients as an Hermitian form in
$-ds_i^{-1}+id\theta_i$ and $dz_j$ and in this sense has leading part
\begin{equation}
g_{\text{Ri}}=\frac34\sum\limits_{i=1}^k\left(\frac{ds_i^2}{s_i^2}+s_i^2d\theta_i^2\right)+h
\label{MetResMod.390}\end{equation}
where $h$ is log-smooth and restricts to the exceptional divisor to be the
induced Ricci metric.
\end{theorem}

\begin{proof} In terms of the coordinates in~\eqref{MetResMod.412}, the
  Weil--Petersson metric $g_{\WP}$ is as in \eqref{MetResMod.301}. Computed in terms of the these complex differentials, the determinant of
the metric takes the form
\begin{equation}
\begin{gathered}
\det(g_{\WP})
=(\prod_{i=1}^{k} \pi\frac{1}{(\log |z_{i}|)^{3}|z_{i}|^{2}} )\det(g_{\WP,z})\left(1+
\sum\limits_{i=1}^k(\ilog|s_{i}|)f_{i}\right)\\
\log\det(g_{\WP})
=C+\sum\limits_{i=1}^k -3i\log\log|z_{i}|  +\log \det g_{\WP,z} +
\sum_{i} \ilog|s_{i}|\tilde f_{i} 
\end{gathered}
\label{MetResMod.393}\end{equation}
where the $f_{i}$ and $\tilde f_{i}$ are log-smooth with respect to the $z_{i}$ variable.

Since $-\log\det(g_{\WP})$ is a K\"ahler potential for the Ricci metric the
leading normal part of the metric is
\begin{equation}
\frac{3}{4}\sum\limits_{i} \frac{|dz_i|^2}{(\log|z_{i}|)^{2}|z_{i}|^{2}}
\label{MetResMod.395}
\end{equation}
Changing variables back to  $s_{i}=-1/\log|z_{i}|$ gives
\eqref{MetResMod.390} with the constants matching  Corollary 4.2 of~\cite{MR2169586}.
\end{proof}

Thus the Ricci metric of the Weil--Petersson metric is of the form of a
`multi-cusp' metric, as exemplified by the product of Riemann surfaces with
cusps. It is shown in~\cite{MR2169586} that the K\"ahler-Einstein metric on
the moduli space is quasi-isometric to the Ricci metric; it presumably
has similar regularity although this has not been demonstrated. Since such
metrics also appear in the setting of locally symmetric spaces it is very
natural to enquire as to the structure of the continuous spectrum in these settings.

\section{Sectional curvature}

Written out in terms of the coordinate introduced by the real resolution as
in~\eqref{MetResMod.412}, the Weil--Petersson metric is given by the
Hermitian form
\begin{equation}
\begin{gathered}
g_{\WP}=\sum g_{j\bar l}\alpha _j\overline{\alpha _{l}},\\
g_{j\bar l}= \begin{cases}\pi s_j^3(1+s _j^2\gamma_{j\bar j})& j=l\le k\\
s_j^3s_l^3\gamma_{j\bar l}& j\not=l\le k\\
s_j^3\gamma _{j\bar l}&j\le k,\ l>k\\
\gamma _{j\bar l}& j,l>k
\end{cases}
\end{gathered}
\label{20.8.2016.1}\end{equation}
where the $\gamma _{j\bar l}$ are log-smooth, have $\theta_j$ derivatives
vanishing to infinite order at $s_j=0$ and for $j,l>k$, $g_{j\bar l}$
restricts to the corner $\cap_{j\le k}\{s_j=0\}$ to give the induced
Weil--Petersson metric. The structure of the metric comes directly from the
formula~\eqref{MetResMod.301} which shows that the co-metric with respect
to the dual basis is log-smooth, with invertible tangential part, except
for the diagonal components in the `normal' directions which are of the form
\begin{equation}
g^{j\bar j}=\pi^{-1}s_j^{-3}(1+s_j^2\delta _j),\ 1\le j\le k
\label{20.8.2016.3}\end{equation}
with $\delta _j$ log-smooth. The Fenchel--Nielsen coordinates are geodesic coordinates~\cite{MR0204641}, therefore restricting at $\{s_{j}=0\}$ the cross terms $g_{j\bar l}$, $j \neq l$, vanish there. Note that this is different from the plumbing coordinates used in~\eqref{MetResMod.412}. This implies that the K\"ahler potential of the metric near the intersection of $k$ divisors is given by a log-smooth function with $\theta_{j}$ derivatives vanishing to infinite order at $s_{j}=0$ and with expansion
\begin{equation}
u(s_{1},\dots, s_{k}, z_{k+1}, \dots, z_{3g-3+n})=\psi(z,\bar z)+\sum_{i=1}^{k}2\pi s_{i}+\sum_{i=1}^{k}s_{i}^{2}\phi_{i}(s,z,\bar z),
\label{MetResMod.403}\end{equation}
where $\phi_{i}=\sum_{j\neq i} O(s_{j}^{2})$. The Weil--Petersson metric, given by the K\"ahler form $g=\partial \dbar u$, is 
\begin{multline}
g=\begin{pmatrix}
\pi s_{i}^{3}+\frac32 s_{i}^{4}\phi_{i} + O(s_{i}^{5}) & O(s_{i}^{3}s_{j}^{3})  & s_{i}^{3}\phi_{i,\bar z}\\
O(s_{i}^{3}s_{j}^{3}) & \pi s_{j}^{3}+\frac32 s_{j}^{4}\phi_{j} + O(s_{j}^{5})
&s_{j}^{3}\phi_{j,\bar z} \\
s_{i}^{3}\phi_{i,z} & s_{j}^{3}\phi_{j, z } & \psi_{z\bar z}+ O(\sum_{i} s_{i}^{2}) 
\end{pmatrix}, \\
1\leq i \neq j \leq k
\label{MetResMod.404}\end{multline}
with the dual metric 
\begin{equation}
g^{-1}=\begin{pmatrix}
\pi^{-1}s_{i}^{-3}+O(s_{i}^{-2}) & O(1)& O(1)\\
O(1) & \pi^{-1}s_{j}^{-3}+O(s_{j}^{-2}) &O(1)\\
O(1) &O(1) & \psi_{z\bar z}^{-1}+\sum O(s_{i}^{2})
\end{pmatrix}.
\label{MetResMod.405}\end{equation}
  
 \begin{proposition}\label{MetResMod.408}
The leading order of the curvature tensors of the Weil--Petersson metric near the intersection of $k$ divisors $\cap_{j=1}^{k}\{s_{j}=0\}$ in the orthonormal basis $\{\frac{ds_{j}}{s_{j}^{\frac12}}+s_{j}^{\frac32}id\theta_{j}, \ dz_{l},\  1\leq j \leq k, \ l>k\}$ 
 are given by the matrices below with entries $(q,p) \in \{i,j,*\} \times \{i,j,*\}$:
 \begin{equation}
\begin{gathered}
\tilde R_{i\bar i q\bar p}
=O\begin{pmatrix}
s_{i}^{-1} & s_{i}^{\frac12} & s_{i}^{\frac12}\\
s_{i}^{\frac12} & s_{i}^{\frac12}s_{j}^{\frac12} & s_{i}^{\frac12}\\
s_{i}^{\frac12} & s_{i}^{\frac12} & s_{i}
\end{pmatrix}
\\
\tilde R_{i \bar j q \bar p}
=O\begin{pmatrix}
s_{i}^{\frac12} & s_{i}s_{j} & s_{i}\\
s_{i}s_{j} & s_{j}^{\frac12} & s_{j}\\
s_{i} & s_{j} & s_{i}^{\frac32}s_{j}^{\frac32}
\end{pmatrix},\
\tilde R_{i \bar * q \bar p}
=O\begin{pmatrix}
s_{i}^{\frac12} & s_{i} & s_{i}^{2}\\
s_{i} & s_{j}^{\frac12} & s_{i}^{\frac32}s_{j}^{\frac32}\\
s_{i}^{2} & s_{i}^{\frac32}s_{j}^{\frac32} & s_{i}^{\frac32}
\end{pmatrix}\\
\tilde R_{*\bar i q \bar p}
=O\begin{pmatrix}
s_{i}^{\frac12} & s_{i} & s_{i}^{2}\\
s_{i} & s_{j}^{\frac12} & s_{i}^{\frac32}s_{j}^{\frac32} \\
s_{i}^{2} & s_{i}^{\frac32}s_{j}^{\frac32} & s_{i}^{\frac32}
\end{pmatrix},\
\tilde R_{*\bar * q \bar p}
=O\begin{pmatrix}
s_{i} & s_{i}^{\frac32}s_{j}^{\frac32} & s_{i}^{\frac32} \\
s_{i}^{\frac32}s_{j}^{\frac32} & s_{j} & s_{j}^{\frac32}\\
s_{i}^{\frac32} & s_{j}^{\frac32} & 1
\end{pmatrix}
\\
(q,p) \in \{i,j,*\} \times \{i,j,*\},\ 
1\leq i,j \leq k, \ k+1\leq * \leq 3g-3+n,
\end{gathered}\label{MetResMod.407}\end{equation}
where more specifically the sectional curvature of the normal direction $\tilde R_{i\bar i i \bar i}$ is given by 
\begin{equation}
\tilde R_{i\bar i i \bar i}=-\frac{3\pi}{4}s_{i}^{-1}+O(s_{i}).
\label{MetResMod.409}\end{equation}
 \end{proposition}
 \begin{remark}
 The leading coefficient in~\eqref{MetResMod.409} matches with the  scalar curvature of the leading term in~\eqref{MetResMod.301} which gives
 $$
 R=-\frac{3\pi}{2}s_{i}^{-1}.
 $$
 \end{remark}

\begin{proof}

We show the computation of the case with one single divisor first $k=1$, and the computation with multiple divisors are similar. Consider the K\"ahler potential $u(s,z,\bar z)$ for the Weil--Petersson metric, which should be
\begin{equation}\label{MetResMod.413}
u=\phi(z, \bar z) + 2\pi s+s^{2}\psi(s, z,\bar z)
\end{equation}
the metric $g_{i\bar j}$ is given by (here $\pa_{\alpha}=\frac{1}{2}(s^{2}\pa_{s}-i\pa_{\theta}$) and similarly $\pa_{\bar \alpha}=\frac{1}{2}(s^{2}\pa_{s}+i\pa_{\theta})$)
\begin{equation}\label{MetResMod.414}
g_{i\bar j}=\pa_{i}\pa_{\bar j}u
=\begin{pmatrix}
u_{\alpha\bar \alpha}& u_{\alpha\bar z}\\
u_{\bar \alpha z}& u_{z \bar z}
\end{pmatrix}
=\begin{pmatrix}
\pi s^{3} + \frac32s^{4}\psi + \frac32 s^{5}\psi_{s}+\frac14 s^{6}\psi_{ss}& s^{3}\psi_{\bar z}+\frac12 s^{4}\psi_{s\bar z}\\
 s^{3}\psi_{z}+\frac12 s^{4}\psi_{sz} & \phi_{z\bar z}+s^{2}\psi_{z\bar z}
\end{pmatrix}.
\end{equation}
Note the
metric itself, is of the form 
\begin{equation}
\begin{pmatrix}
s^3(\pi+a's)& s^3b'\\
s^3\overline{b'}& h'
\end{pmatrix}
\label{MetResMod.397}\end{equation}
where $a',$ $b'$ and $h'$ are again log-smooth and $h'$ is invertible.

The dual metric (always using the b-basis which
becomes the cusp basis) is of the form 
\begin{equation}
\begin{gathered}
g^{-1}=(\det g_{i\bar j})^{-1}
\begin{pmatrix}
u_{z \bar z}& -u_{\alpha\bar z}\\
-u_{\bar \alpha z}& u_{\alpha\bar \alpha}
\end{pmatrix}\\
= \begin{pmatrix}
\pi^{-1}s^{-3}-\frac32 s^{-2}\pi^{-2}\psi & -\frac{1}{\pi}\phi_{z\bar z}^{-1}\psi_{\bar z}\\
 -\frac{1}{2\pi}\phi_{z\bar z}^{-1}\psi_{z} & \phi_{z\bar z}^{-1}
\end{pmatrix}
+
O \begin{pmatrix}
s^{-1} & s\\
s & s^{2}
\end{pmatrix}
\end{gathered}
\label{MetResMod.396}\end{equation}

From here we compute the curvature tensor:
\begin{equation}
\begin{gathered}
R_{1\bar 1 q\bar p}=-\begin{pmatrix}
\frac{3\pi}{4} s^{5}+O(s^{6}) & O(s^{5})\\
O(s^{5}) & \frac32s^{4}\psi_{z\bar z}+O(s^{5})
\end{pmatrix}
\\
R_{*\bar 1 q\bar p}=-\begin{pmatrix}
\frac34 \psi_{z} s^{5} + O(s^{6}) & O(s^{5})\\
O(s^{5})& O(s^{3})
\end{pmatrix}\\
R_{1\bar * q\bar p}=-\begin{pmatrix}
\frac34 \psi_{\bar z} s^{5}+O(s^{6})&O(s^{5})
\\
O(s^{5})&O(s^{3})
\end{pmatrix}\\
R_{*\bar * q\bar p}=-\begin{pmatrix}
\frac32 \psi_{z\bar z}s^{4} + O(s^{5}) &O(s^{3})\\
O(s^{3}) & O(1)
\end{pmatrix}
\end{gathered}
\label{MetResMod.400}\end{equation}

In the orthonormal basis, we need to rescale the basis by changing from $\frac{ds}{s^2}+id\theta$ to unit length vector
$\frac{ds}{s^{\frac12}}+s^{\frac32}id\theta,$ so effectively multiplying
  each entry with a `$i$' or `$\bar i$' by $s_{i}^{-\frac 32}$. Therefore,
\begin{equation}
\begin{gathered}
\tilde R_{1\bar 1 q\bar p}=
\begin{pmatrix}
-\frac{3\pi}{4}s^{-1}& 0\\
0&0
\end{pmatrix}+O\begin{pmatrix}
1& s^{\frac12}\\
s^{\frac12} & s
\end{pmatrix}\\
\tilde R_{*\bar 1 q\bar p}=O\begin{pmatrix}
s^{\frac12} & s^{2}\\
s^{2} & s^{\frac32}
 \end{pmatrix},\
\tilde R_{1\bar * q\bar p}=O\begin{pmatrix}
s^{\frac12} & s^{2}\\
s^{2}& s^{\frac32}
\end{pmatrix},\
\tilde R_{*\bar * q\bar p}=O\begin{pmatrix}
s & s^{\frac32}\\
s^{\frac32} & 1
 \end{pmatrix}.
\end{gathered}
\label{MetResMod.401}\end{equation}

\end{proof}

\section{Takhtajan-Zograf metrics}\label{TZ}

To analyse the behaviour of the Takhtajan-Zograf metric(s) at the divisors
of $\oMs{g,n},$ $n\ge1,$ we need (in addition to everything above) to see
what happens to the appropriate `Eisenstein series' under degeneration.

If we have a punctured Riemann surface undergoing nodal
degeneration with a fixed cusp $B_i$ we need to understand the behaviour of the
Eisenstein series, which can be realized as the solution to  
\begin{equation}
(\Lap+2)E_i(z)=0,\ E_i(z)=x^{-2}\chi+E'_i,\ E'_i\in L^2
\label{MetResMod.277}\end{equation}
where $\chi$ is a cut-off near $B_i.$ So the behaviour of this follows from
the analysis of $(\Lap+2)^{-1}$ above.

\begin{lemma}\label{MetResMod.289} On any connected compact Riemann surface
  with punctures, the Eisenstein series $E_i$ determined by
  \eqref{MetResMod.277} for the puncture $B_i$ is strictly positive.
\end{lemma}

\begin{proof} Maximum principle.
\end{proof}

For this Eisenstein series associated with any one marked point on a
marked Riemann surface the limit at another marked point 
\begin{equation}
L(i,p)=\lim \rho_{p}^{-1}E_i
\label{MetResMod.293}\end{equation}
is the value of some L-function.

Now consider the degeneracy of $E_i$ at some $k$-fold (self-)intersection
of divisors at the boundary of $\oMs{g,n}.$ The nodal Riemann surface
corresponds to at most $k$ components of connected Riemann surfaces $M_j$ of
genus $g_j.$ Each component has (possibly empty) finite set of $n_j+k_j$
points, consisting of $n_j$ labelled marked
points and $k_j$ distinct but unlabelled nodal points; each component is stable in the sense
that $2g_j+n_j+k_j>2$ and there is an involution pairing the nodal points
(collectively). At each boundary point, the b-cotangent bundle of
$\Ms{g,n}$ contains as a subspace the b-cotangent bundle of $\Ms{g_j,n_j+k_j}$.

Now, for such a nodal Riemann surface in the boundary of the moduli space,
the $i$th marked point appears in precisely one of the components
corresponding to $j=m(i).$ Each of the other components is connected to
this particular component by a chain of separating nodes. Let $s(m(i),j)$
be the number of these nodes and let $n(i,j)$ formally denote the last
node which is in the $j$th component Riemann surface. We also let
$L(m(i),j)$ be the product of the L-functions corresponding to the chain
of intervening `linking' nodes.

\begin{proposition}\label{MetResMod.290} Each Eisenstein series $E_i$ is
  log-smooth on $\hCs{g,n}$ (the metric resolution
  of the universal curve) and at a $k$-fold corner is of the form 
\begin{equation}
E_i=\sum\limits_{j}\rho_{i}^{4\sigma (m(i),j)}L(m(i),j)E_{n(i,j)}(\Ms{g_j,n_j+k_j}).
\label{MetResMod.291}\end{equation}
\end{proposition}

For one of the `fixed divisors' in $\oMs{g,n},$ $n>0,$ the corresponding
Takhtajan-Zograf metric is the push-forward in
\begin{equation}
(q_1,q_2)_{\TZ}=\phi_*(\frac{E_i^{-1}q_1\bar q_2}{\mu_H})
\label{MetResMod.278}\end{equation}
and the total metric is the sum over $i;$ see \cite{MR2443304}.

The extra factor is therefore at worst $O(x^{2})$ in terms of the
logarithmic coordinates, so does not affect the exponential decay from the
vanishing (i.e.\ simple pole) of the quadratic differentials at the cusp
face with which it is associated. In fact the final effect is that there is
one extra order of singularity at the nodal face relative to the Weil--Petersson
metric. The leading term can be extracted as before, except that there is
an overall constant which is global in nature and the Takhtajan-Zograf
metric will have the behaviour 
\begin{equation}
g_{\TZ}=cds^2+c's^4d\theta^2
\label{MetResMod.280}\end{equation}
in the normal direction to the divisor. According to Obitsu-To-Weng there
is some degeneracy in the tangential directions, see~\cite{MR2443304}.

\section{Lengths of short geodesics}\label{SG}

For the plumbing model the shortest geodesic occurs in the middle of the
hyperbolic neck, that is, at $|z(\theta)|=\sqrt{|t|}$ in terms of the
original complex coordinates. The length of this circle is $2\pi^{2}s,$ $s=\ilog|t|.$
This provides an approximation to the degenerating geodesic for the global
hyperbolic metric both in terms of the length and the position of the circle.

\begin{proposition} In terms of the local fiber coordinate
$w=\ilog|z|/s$ near the front face of the metric resolution the
short closed geodesic near a nodal point is of the form 
\begin{equation}
\gamma_s(\theta)=(w(s,\theta), \theta)=(2+g(s)+g'(s,\theta),\theta)
\label{MetResMod.311}\end{equation}
where $g(s)$ is log-smooth with $g(0)=0$ and $g'(s,\theta)$ is smooth and vanishes to
infinite order as $s\downarrow0;$ it follows that its length is
\begin{equation}
L_\gamma(s)=2\pi^{2}s(1+se(s))
\label{MetResMod.313}\end{equation}
where $e$ is log-smooth.
\end{proposition}

\begin{proof} On the metric resolution, near a nodal point, the
  degenerating hyperbolic metric, takes the form 
\begin{equation}
    g=e^{2s^2f}Z(w)^2\left(\frac{dw^{2}}{w^2}+w^2s^2d\theta^{2}\right),\
Z(w)=\frac{\pi/w}{\sin(\pi/w)}
\label{MetResMod.501}\end{equation}
where $f=f(w,s,\theta))$ is log-smooth and $\pa_\theta f=O(s^{\infty}).$ 

To show that the actual degenerating geodesic is close to the curve for the
model, $\gamma(\theta)=(2,\theta),$ consider the length of families of curves of the form 
\begin{equation}
\gamma(\theta,s)=(w(s,\theta),\theta)=(2+h(s)+su(\theta,s),\theta),
\label{MetResMod.309}\end{equation}
where $h\in H^1([0,1))$ and $u\in H^1_0(\bbS\times[0,1))$ lies in the Hilbert subspace
without constant term, so $\int ud\theta=0.$ Then the length of the family satisfies
\begin{equation}
\frac{L(\gamma)}{s}=
\int
e^{s^2f(\gamma)}Z(w)E^{\ha}d\theta,\ E=\frac{(u')^2}{w^{2}}+w^2.
\label{MetResMod.310}\end{equation}
This is a $C^2$ function near zero in $[0,1)_s\times H^1([0,1))\times
H^1_0(\bbS\times[0,1))$ which for small $s$ has a non-degenerate minimum at
$(h,u)=0.$ Here the smoothness uses the fact $w>0$ and $E>0,$ so its inverse and 
square-root are strictly positive functions in $L^{\infty}.$  
Indeed, the derivative with respect to $(h,u)$ evaluated on the tangent vector
$(\kappa,\upsilon)$ may be written in the form
\begin{equation}
\begin{gathered}
\int \kappa \alpha +\upsilon'\beta \\
\alpha =e^{s^2f}Z(w)E^{\ha}
\left(s^2f_{w}-\frac{1}{w}(1-\frac{\pi}{w}\cot \frac{\pi}{w})
+E^{-1}(\frac{(u')^{2}}{w^{3}}+w) \right)\\
\beta =e^{s^2f}Z(w)E^{-1/2}\frac{u'}{w^{2}}-s\int\alpha 
\end{gathered}
\label{MetResMod.313a}\end{equation}
where the last integral is the unique in $\theta$ without constant term. So
using $L^2$ duality this may be identified as a map, rather than a linear form,
\begin{multline}
D_{s}:  H^1([0,1)\times L^2_0(\bbS;H^1([0,1))\ni(h,u)\longmapsto \\
(\alpha, \beta)\in H^1([0,1))\times L^2_0(\bbS;H^1([0,1))
\end{multline}
where the $0$ subscript indicates the absence of the constant mode.
As such it is again $C^2$ and its derivative at $s=0$ is invertible. From the
Implicit Function Theorem, $D_s$ has a unique $0$ near $0$ and from this the
stated regularity of the curve giving the unique geodesic follows
directly. The log-smoothness of the length can then be seen by evaluating
the integral as a push-forward.
\end{proof}

\appendix\section*{Appendix: Log-smooth functions}

In this appendix we recall the definition, and some of the basic
properties, of log-smooth conormal functions on a manifold with corners.

For a general compact manifold with corners the space of log-smooth
functions, denoted by $\CI_{\log}(M)$, is well defined in terms of iterated
expansions at each of the boundary faces. Proceeding by induction one can
suppose that $\CI_{\log}(M)$ is well-defined for any manifold with faces of
codimension at most $k.$ Then on a manifold $N$ with boundary faces of
codimension up to $k+1$ a function $u\in\CI(N\setminus\pa N)$ is in
$\CI_{\log}(N)$ if it has expansions at each boundary hypersurface
$H=\{\rho =0\}$ with defining function $\rho$ and for any (one) choice of
collar decomposition $\{\rho <\epsilon \}=H\times[0,\epsilon):$
\begin{equation}
u\simeq\sum\limits_{l=0}^{\infty}
\sum\limits_{j=0}^lu_{l,j}(\log\rho)^j\rho^l,\ u_{l,j}\in\CI_{\log}(H) 
\label{MetResMod.198}\end{equation}
where $H$ can have boundary faces only up to codimension $k.$ The precise
meaning of the expansion can be given in terms of conormal
estimates. Namely if $\cA(N)$ is the space of bounded conormal functions,
defined by the stability condition
\begin{equation}
\Diffb m(N)\cdot u\subset L^{\infty}(N)\ \forall\ m
\label{MetResMod.99}\end{equation}
then the remainder terms in \eqref{MetResMod.198} are required to satisfy 
\begin{equation}
\phi(\rho)\left[u- \sum\limits_{l=0}^{M}
\sum\limits_{j=0}^lu_{l,j}(\log\rho)^j\rho^l\right]\in \rho
^{M+1}\cA(N)\ \forall\ M
\label{MetResMod.100}\end{equation}
where $\phi\in\CIc(\bbR)$ has support in the collar and $\phi=1$ near $0.$

Such an expansion at a boundary face implies similar expansions near the
corners contained in it. Indeed, again proceeding by induction over boundary
codimension, the expansion of the coefficients in \eqref{MetResMod.198}
gives an expansion at any boundary face $F$ of $H$ of the form
\begin{equation}
u\simeq\sum\limits_{\alpha,\beta \le\alpha }^{\infty}u_{\alpha
  ,\beta}(\log\rho)^\beta \rho^\alpha ,\ u_{\alpha ,\beta }\in\CI_{\log}(F)
\label{MetResMod.101}\end{equation}
where $\rho$ now stands for the $\codim(F)$ defining functions of $F,$
one of which is by assumption a defining function for $N.$ Again the
meaning of this expansion is that the difference of $u$ and the terms with
$|\alpha |\le L$ should lie in $R^{L-1}_F\cA(M)$ where $R_{F}$ is a radial
defining function for $F.$ The function $u$ determines all the coefficients
in the expansion at any boundary face and it follows that there are
compatibility conditions across the higher codimension faces.

These compatibility conditions are between the expansions at different
boundary faces but the expansion at any one boundary face is
unrestricted. This can be seen by constructing appropriate elements of
$\CI_{\log}(M).$ The series in \eqref{MetResMod.101} can be summed by
choosing a cutoff $\chi\in\CIc(\bbR^p),$ where $p$ is the codimension and
$\chi=1$ in a neighbourhood of the origin. Then, provided
$\epsilon_l\downarrow0$ converges sufficiently rapidly, depending on the
coefficients $u_{\alpha ,\beta }\in\CI_{\log}(F),$
\begin{equation}
u=\sum\limits_{l}\sum\limits_{|\alpha|=l,\beta \le\alpha }^{\infty}u_{\alpha
  ,\beta}(\log\rho)^\beta \rho^\alpha\chi(\frac{\rho}{\epsilon_{l}} )\in\CI_{\log}(M)
\label{MetResMod.102}\end{equation}
satisfies \eqref{MetResMod.101}. Moreover, if $u'\in \CI_{\log}(M)$ has the
same expansion at $F$ then the difference can be decomposed near $F$ as a sum over
the boundary hypersurfaces containing $F$
\begin{equation}
u'-u=u''+\sum\limits_{H\supset F}v_H,\ v_H\in\CI_{\log}(M),\ \supp(u'')\cap
F=\emptyset
\label{MetResMod.103}\end{equation}
where each $v_H$ has a trivial expansion at all faces of codimension two or
higher which are not contained in $H.$ 

Since we construct functions below by iteration over such asymptotic sums
it is useful to consider subspaces of $\CI_{\log}(M)$ for which the
expansions at a given collection of boundary faces are trivial. We use the
notation $u\overset F\equiv0$ to indicate that the expansion at the boundary
face $F$ is trivial.

\begin{lemma}\label{MetResMod.104} If $u\in\CI_{\log}(M)$ and $u\overset
  F\equiv0$ at all boundary faces of codimension $k$ then $u$ may be
  decomposed into a sum over boundary faces \{G\} of codimension $k-1$ 
\begin{equation*}
u=\sum\limits_{G}u_G,\ u_G\in\CI_{\log}(M),\ u_G\overset F\equiv0.
\label{MetResMod.105}\end{equation*}
\end{lemma}

\begin{proof} If $\ff$ is a boundary face of $M$ under the blow up of $F$
 the space $\{u\in\CI_{\log}(M);u\overset F\equiv0\}$ lifts isomorphically
 to $\{v\in\CI_{\log}([M;F]);v\overset{\ff}\equiv0\}.$ After the blow up
 of all boundary faces of codimension $k,$ the lifts of the faces of
 codimension $k-1$ are disjoint. Thus, on the blown up space the lift of
 $u$ can be divided into pieces each of which has support disjoint from one
 of the (lifted) boundary faces of codimension $k-1$ by use of a partition
 of unity. These pieces therefore are the lifts of a decomposition as
 desired. 
\end{proof}

\bibliography{MetResMod}
\bibliographystyle{amsplain}

\end{document}